\documentclass{amsart}
\usepackage{amsmath, amscd, amssymb, amsthm}
\usepackage{bbm}
\usepackage{latexsym}
\usepackage{amsfonts}
\usepackage{graphicx}
\usepackage[all,cmtip]{xy}
\usepackage[colorlinks,linkcolor=blue,breaklinks=blue,urlcolor=blue,citecolor=blue,anchorcolor=blue,pagebackref]{hyperref}%
\usepackage{geometry}
\newtheorem{theorem}{Theorem}
\newtheorem{lemma}{Lemma}

\newtheorem{remark}{Remark}
\newtheorem{proposition}{Proposition}

\theoremstyle{definition}
\newtheorem{definition}{Definition}

\renewcommand*\backref[1]{}
\renewcommand*\backrefalt[4]{ \ifcase #1 \or (cited on page #2) \else (cited on pages #2) \fi}

\newcommand{\be}{\begin{equation}}
\newcommand{\ee}{\end{equation}}
\newcommand{\bea}{\begin{eqnarray}}
\newcommand{\eea}{\end{eqnarray}}

\newcommand{\eps}{\varepsilon}

\newcommand{\vs}{\vspace{0.5cm}}

\newcommand{\vsv}{\vspace{0.12cm}}

\def\XXint#1#2#3{{\setbox0=\hbox{$#1{#2#3}{\int}$ }
\vcenter{\hbox{$#2#3$ }}\kern-.6\wd0}}

\begin{document}

\title[Bismut K\"ahler-like manifolds]{Bismut K\"ahler-like manifolds of dimension $4$ and $5$}
	
\author{Quanting Zhao}
\address{Quanting Zhao. School of Mathematics and Statistics, and Hubei Key Laboratory of Mathematical Sciences, Central China Normal University, P.O. Box 71010, Wuhan 430079, P. R. China.} \email{zhaoquanting@126.com;zhaoquanting@mail.ccnu.edu.cn}
\thanks{Zhao is partially supported by National Natural Science Foundations of China with the grant No.11801205 and 12171180.
Zheng is partially supported by National Natural Science Foundations of China
with the grant numbers 12141101 and 12071050, Chongqing grant cstc2021ycjh-bgzxm0139, and is supported by the 111 Project D21024.}

\author{Fangyang Zheng}
\address{Fangyang Zheng. School of Mathematical Sciences, Chongqing Normal University, Chongqing 401331, China}
\email{20190045@cqnu.edu.cn} \thanks{}


\begin{abstract}
This paper is a sequel to our studies \cite{ZZ} and \cite{YZZ} on Bismut K\"ahler-like manifolds, or {\em BKL} manifolds for short. We will study the structural theorems for {\em BKL} manifolds, prove a conjecture raised in \cite{YZZ} which states that any {\em BKL} manifold that is Bismut Ricci flat must be Bismut flat, and give complete classifications of {\em BKL} manifolds in dimension $4$ and $5$.
\end{abstract}

\maketitle

\tableofcontents

\markleft{Quanting Zhao and Fangyang Zheng}
\markright{Bismut K\"ahler-like manifolds}

\section{Introduction and statement of results}

Recall \cite{Bismut} that the {\em Bismut connection} $\nabla^b$ of a Hermitian manifold $(M^n,g)$ is the unique Hermitian connection with totally skew-symmetric torsion tensor. It is also called {\em Strominger connection} in some literature due to the seminal work \cite{Strominger} where it was called the H-connection. In physics literature it is often called the {\em KT} (K\"ahler with torsion) connection or {\em characteristic} connection. Since the need of non-K\"ahler Calabi-Yau spaces in string theory (\cite{GatesHR}, \cite{Strominger}), this connection has been receiving more and more attention from geometers and mathematical physicists alike. We refer the readers to  \cite{AI}, \cite{AV}, \cite{EFV}, \cite{FG},\cite{FY}, \cite{FGV}, \cite{FinoTomassini}, \cite{FV}, \cite{FV1}, \cite{Fu}, \cite{Fu-Li-Yau}, \cite{GGP}  \cite{IvanovP}, \cite{Liu-Yang},   \cite{Popovici}, \cite{Sal}, \cite{S18}, \cite{STW}, \cite{Tosatti}, \cite{Tseng-Yau}, \cite{Ugarte}, \cite{ZZ3}, \cite{ZhouZ} and the references therein for more discussions on Bismut-Strominger connection, pluriclosed metric, and related topics.

Yang and Zheng \cite{YZ} introduced the notion of K\"ahler-likeness for Levi-Civita and Chern connections, following the pioneer work of Gray \cite{Gray} and others. Angella, Otal, Ugarte, and Villacampa \cite{AOUV} generalized it to all metric connections on Hermitian manifolds (see also \cite{FT}, \cite{FTV}). A metric connection $D$ on $(M^n,g)$ is said to be {\em K\"ahler-like,} if its curvature tensor $R^D$ satisfies
\begin{eqnarray*}
&& R^D(x,y,z,w) + R^D(y,z,x,w) + R^D(z,x,y,w) = 0, \\
&& R^D(x,y,Jz,Jw)= R^D(x,y,z,w) = R^D(Jx,Jy,z,w),
\end{eqnarray*}
for any tangent vectors $x$, $y$, $z$, $w$ in $M^n$, where $J$ is the almost complex structure of $M$. The first equation is  called the {\em Bianchi identity}, and the second line is called the {\em type condition}. Note that the first equality in the type condition is always satisfied when $DJ=0$.

If we extend the metric $g=\langle , \rangle$ and the curvature $R^D$ linearly over ${\mathbb C}$, and use the decomposition $TM\otimes {\mathbb C} = T^{1,0}M \oplus T^{0,1}M$, where $T^{1,0}M$ consists of all vector fields of the form $X= x-\sqrt{-1}Jx$, where $x$ is real, then the above definition, as in \cite[Remark 6]{AOUV}, is equivalent to
\begin{eqnarray*}
&& R^D(X,\overline{Y}, Z, \overline{W}) = R^D(Z,\overline{Y}, X, \overline{W}), \\
&& R^D(X, Y, \ast , \ast )= R^D(\ast , \ast , Z, W) = 0,
\end{eqnarray*}
for any type $(1,0)$ complex tangent vectors $X$, $Y$, $Z$, and $W$.

For each of the three canonical connections on Hermitian manifolds, namely, Levi-Civita (or Riemannian), Chern, and Bismut connection, there are compact K\"ahler-like manifolds that are not K\"ahler. For instance, any compact Chern flat manifold \cite{Boothby} is certainly Chern K\"ahler-like, and there are non-K\"ahler such examples in dimension $n\geq 3$. However, compact Chern K\"ahler-like manifolds are known to be Chern flat within some special classes (see for example \cite{GZ}, \cite{ZZ1}, \cite{Zheng1}), and to our best knowledge, there is no known example yet of a compact non-K\"ahler manifold which is Chern K\"ahler-like but not Chern flat, despite the general belief that such manifolds should exist.

For Riemannian K\"ahler-like ones, any Riemannian flat Hermitian manifold will be such examples, and in complex dimension $n\geq 3$, there are plenty of such metrics that are non-K\"ahler, yet a full classification is still missing when $n\geq 4$, although the three dimensional case has been solved by \cite{KYZ}. When $n\geq 3$, there are also examples of compact non-K\"ahler manifolds which are Riemannian K\"ahler-like but not Riemannian flat.

For $t$-Gauduchon connections \cite{Gauduchon1}, which is the line of connections joining Chern and Bismut, it was proved in the recent work of Lafuente and Stanfield \cite{LS} that other than Chern or Bismut, a Gauduchon connection cannot be K\"ahler-like unless the metric is already K\"ahler. Earlier partial results were obtained in \cite{Fu-Zhou}, \cite{VYZ}, \cite{YZ1}, \cite{ZZ2}.

The richest class of non-K\"ahler, K\"ahler-like examples occur for Bismut connection. We will call a Hermitian manifold $(M^n,g)$ whose Bismut connection is K\"ahler-like a {\em Bismut K\"ahler-like} manifold, or a {\em BKL} manifold in short. (Note that it was also called Strominger K\"ahler-like manifolds in some literature). Compact, non-K\"ahler {\em BKL} manifolds form a rather interesting and restrictive class of Hermitian manifolds. The special case when the Bismut connection is flat (or {\em Bismut flat} in short) was classified in \cite{WYZ}. They are exactly the compact quotients of Samelson spaces \cite{Samelson}, namely, simply-connected Lie groups with bi-invariant metrics and compatible left invariant complex structures. In \cite{AOUV}, the authors classified all the {\em BKL} manifolds amongst all complex nilmanifolds and Calabi-Yau type solvmanifolds of dimension $n=3$.  We showed in \cite{ZZ} that all {\em BKL} manifolds are pluriclosed, confirming a conjecture raised in \cite{AOUV}. What we proved in \cite{ZZ} is actually slightly stronger, namely, given any Hermitian manifold, the {\em BKL} condition is equivalent to the pluriclosedness of the metric plus the torsion parallelness with respect to Bismut connection.

When $n=2$, the {\em BKL} condition is equivalent to the condition {\em Vaisman}, which means a Hermitian manifold that is locally conformal K\"ahler and its Lee form is parallel under Levi-Civita connection. Compact Vaisman surfaces were classified by Belgun in the beautiful work \cite{Belgun}.

In our recent work \cite{YZZ}, we were able to classify complete non-K\"ahler {\em BKL} threefolds. They are either the product of a K\"ahler curve and a {\em BKL} surface, or the standard Hermitian structure on the product of two Sasakian $3$-manifolds. It turns out that the torsion of {\em BKL} manifolds in dimension $2$ or $3$ are degenerate, namely, $T^{\ast}_{ik}=0$ for any $i,k<n$, where $e_n$ lines up with the direction of the Gauduchon $1$-form $\eta$ (\cite{Gauduchon}). Moreover, we proved in \cite{YZZ} that if a complete, non-K\"ahler {\em BKL} manifold $(M^n,g)$ has degenerate torsion, then its universal cover is the product of a {\em BKL} surface or threefold with a K\"ahler manifold.

For {\em BKL} manifolds with $n\geq 4$, its torsion tensor is no longer always degenerate, unlike in dimensions $2$ and $3$. However, this type of manifolds still seem to be quite restrictive, as illustrated by the numerous properties they must satisfy as proved in \cite{ZZ} and \cite{YZZ}. The goal of this paper is to give a more detailed analysis on {\em BKL} manifolds in general dimensions. We obtain some refined structural results for such manifolds, which in particular enable us to give a complete classification in dimension $4$ and $5$.

To state our main results, first let us recall the tensors $A$, $B$ and $\phi$ on a Hermitian manifold $(M^n,g)$, where more details will be given in \S \ref{tor}. Under any local unitary frame $e$ with dual coframe $\varphi$, they are defined respectively by
$$ A_{i\overline{j}} = \sum_{q,k} T^q_{ik} \overline{ T^q_{jk}}, \ \ \ \ B_{i\overline{j}} = \sum_{q,k} T^j_{qk} \overline{T^i_{qk}},  \ \ \ \ \phi_i^j = \sum_q T^j_{iq} \overline{\eta}_q, $$
where $T_{qk}^i$ are the components of Chern torsion $T^c$ and $\eta_q$ are the ones of the Gauduchon's torsion $1$-form $\eta = \sum_q \eta_q \varphi^q = \sum_{i,q}T^i_{iq}\varphi^q$. The $(1,0)$-type vector field $X_{\eta}$ associated to $\eta$ is defined by the equality $g(Y,\overline{X_{\eta}})=\eta(Y)$ for any $Y$. The tensors $A$ and $B$ are positive semi-definite, and $A$, $B$, $\phi$ are all parallel with respect to the Bismut connection $\nabla^b$ on {\em BKL} manifolds since $\nabla^bT=0$ holds. The rank of $B$, namely the number of positive eigenvalues of $B$, will be called the {\em $B$-rank} of $g$, denoted by $r_{\!{\tiny B}}$, introduced in Definition \ref{cmpt}. It satisfies $0\leq r_{\!{\tiny B}} \leq n-1$, and $r_{\!{\tiny B}}=0$ holds if and only if $g$ is K\"ahler. We have the following

\begin{theorem} \label{thm1}
Let $(M^n,g)$ be a BKL manifold with $n\geq 2$. The condition $A>0$ implies its $B$-rank $r_{\!{\tiny B}}\geq \frac{n}{2}$. If $g$ is complete, then $A\not>0$ if and only if $M^n$ admits at least one K\"ahler de Rham factor. It also holds that $\mbox{ker}A \subseteq \mbox{ker}B \cap X_{\eta}^{\bot}$ and $\mbox{ker}\,\phi =\mbox{ker}B$.
\end{theorem}

For the sake of simplicity, we will call a {\em BKL} manifold {\em full} if $A>0$ in Definition \ref{irr}, so a complete {\em BKL} manifold is full if and only if it admits no K\"ahler de Rham factor. When a {\em BKL} manifold is not full, the discussion could be reduced to the lower dimensional cases, hence we will basically focus on full ones, where it holds that $\frac{n}{2} \leq r_{\!{\tiny B}} \leq n-1$. The next result studies the case when the rank is maximal.

\begin{theorem} \label{thm2}
Let $(M^n,g)$ be a BKL manifold with $n\geq 4$. If $r_{\!{\tiny B}}=n-1$, then $g$ is Bismut flat and full.
\end{theorem}

Recall that a Hermitian manifold is said to be Calabi-Yau with torsion, or {\em CYT} in short, if its Bismut connection $\nabla^b$ has $SU(n)$ holonomy, that is, if the (first) Bismut Ricci curvature is zero.
A conjecture raised in \cite[Conjecture 2]{YZZ} states that, if $(M^n,g)$  is a compact {\em BKL} manifold without K\"ahler de Rham factor of dimension bigger than one, then it cannot be {\em CYT} unless it is Bismut flat. The $n\leq 3$ case of the conjecture was proved in \cite[Theorem 5]{YZZ}. Here we confirm the conjecture in general.

\begin{theorem} \label{thm3}
Let $(M^n,g)$ be a BKL manifold without any K\"ahler de Rham factor of dimension bigger than one. If the (first) Bismut Ricci curvature vanishes, then the Bismut curvature vanishes (that is, $g$ is Bismut flat).
\end{theorem}

\begin{remark}
It also holds from the proof of Theorem \ref{thm3} that, if the (first) Bismut Ricci curvature is semipositive (or seminegative), then the Bismut bisectional curvature is semipositive (or seminegative) for BKL manifolds without any K\"ahler de Rham factor of dimension bigger than one.
\end{remark}

For $q$ non-K\"ahler {\em BKL} surfaces $N_1,N_2,\cdots,N_q$ and an appropriate invertible $q \times q$ matrix $D$, one can define a family of Hermitian metrics, twisted by $D$, on the complex product manifold $\Pi_{i=1}^q N_i = N_1\times\cdots\times N_q$, which are all {\em BKL}. These metrics are not product metrics in general. These Hermitian manifolds will be called the {\em pluriclosed twisted product}, and we will denote them by $N_1 \times_{\! {\tiny D}} \cdots \times_{\! {\tiny D}} N_q$, which will be introduced in Definition \ref{p-twisted} in detail. The study of the minimal $B$-rank case follows, which also gives us the classification of {\em BKL} manifolds of dimension $4$.

\begin{theorem} \label{thm4}
Let $(M^n,g)$ be a full BKL manifold with $n\geq 4$. If $r_{\!{\tiny B}}=\frac{n}{2}$ and $g$ is complete, then the universal cover $\widetilde{M}$ of $M^n$ is the pluriclosed twisted product of $r_{\!{\tiny B}}$ non-K\"ahler BKL surfaces. In particular, for a complete, simply-connected full BKL manifold $(M^4,g)$, it is either Bismut flat, or it is the pluriclosed twisted product of two BKL surfaces.
\end{theorem}

For {\em BKL} manifolds of dimension $5$, we see a strange resemblance to the $3$-dimensional case. It was proved in \cite{YZZ} that any full {\em BKL} threefold is always the standard Hermitan structure on the product of two Sasakian $3$-manifolds. Similarly, one could always construct {\em BKL} fivefolds for any three given Sasakian $3$-manifolds. Let $L_i$ be a Sasakian $3$-manifold, for each $1\leq i\leq 3$. Then there exists a global unit vector field $Y_i$ on the Riemannian $3$-manifold $L_i$, such that $Y_i$ is Killing, namely, $\langle X, \nabla_ZY_i\rangle + \langle Z, \nabla_XY_i\rangle =0$ holds or any vector fields $X$, $Z$, and $\frac{1}{c_i}\nabla Y_i$ gives an orthogonal complex structure on $H_i$ that is compatible with the metric on $L_i$. Here $c_i$ is some positive constant and $H_i$ is the orthogonal complement distribution of $Y_i$ on $L_i$.

Let $M=L_1\times L_2\times L_3 \times {\mathbb R}$ be the Riemannian product manifold, with $g$ the product metric. Denote by $N$ the distribution spanned by $Y_1$, $Y_2$, $Y_3$ and $Y_4$, where $Y_4$ is the (positive) unit vector in the ${\mathbb R}$ factor. For any given $4 \times 4$ skew-symmetric and orthogonal matrix $D=(d_{ij})$, there exists a compatible almost complex structure $J$ on $M$, defined by $D$, namely, $JY_i=\sum_j d_{ij}Y_j$ when restricted on the distribution $N$, while on each $H_i$, $J$ is equal to $\frac{1}{c_i} \nabla Y_i$. That way one gets a Hermitian fivefold $(M, g, J)$. It turns out that this $J$ is always integrable, and the Hermitian metric $g$ is {\em BKL}. We will call them {\em multiple product of Sasakian $3$-manifolds} in Definition \ref{mprod}. Then our classification theorem says that {\em BKL} manifolds in dimension $5$ are essentially of this type:

\begin{theorem} \label{thm5}
Let $(M^5,g)$ be a complete, simply-connected full BKL manifold.  Then either it is Bismut flat, or it is a multiple product of Sasakian $3$-manifolds $L_1\times L_2\times L_3\times {\mathbb R}$.
\end{theorem}

\vs

\section{The torsion of Bismut K\"ahler-like manifolds}\label{tor}
Let $(M^n,g)$ be a Hermitian manifold of complex dimension $n\geq 2$. Denote by $\nabla$, $\nabla^c$, and $\nabla^b$ respectively the Levi-Civita (Riemannian), Chern, and Bismut connection of the metric $g$. Denote by $T^c=T$, $R^c$ the torsion and curvature of $\nabla^c$, and by $T^b$, $R^b$ the torsion and curvature of $\nabla^b$. Under a unitary frame $e$ with dual coframe $\varphi$, the components of $T^c$ are given by
\begin{equation}\label{eq:2.1}
T^c(e_i, \overline{e}_j)=0, \ \ \ T^c(e_i, e_j) = 2 \sum_{k=1}^n T^k_{ij} e_k.
\end{equation}
For Chern connection $\nabla^c$, let us denote by $\theta$, $\Theta$ the matrices of connection and
curvature, respectively, and by $\tau$ the column vector of the torsion $2$-forms, all under the local frame $e$.
Then the structure equations and Bianchi identities are
\begin{eqnarray*}
d \varphi & = & - \ ^t\!\theta \wedge \varphi + \tau,  \\
d  \theta & = & \theta \wedge \theta + \Theta. \\
d \tau & = & - \ ^t\!\theta \wedge \tau + \ ^t\!\Theta \wedge \varphi, \\
d  \Theta & = & \theta \wedge \Theta - \Theta \wedge \theta.
\end{eqnarray*}
The entries of $\Theta$ are all $(1,1)$-forms, while the entries of the column vector $\tau $ are all $(2,0)$-forms.
Similar symbols such as $\theta^b,\Theta^b$ and $\tau^b$ are applied to Bismut connection. The components of $\tau$ are exactly $T_{ij}^k$ defined in (\ref{eq:2.1}):
\[ \tau_k = \sum_{i,j=1}^n T_{ij}^k \varphi_i\wedge \varphi_j \ = \sum_{1\leq i<j\leq n} 2 \ T_{ij}^k \varphi_i\wedge \varphi_j.\]
Under the frame $e$, express the Levi-Civita (Riemannian) connection $\nabla$ as
$$ \nabla e = \theta_1 e + \overline{\theta_2 }\overline{e} ,
\ \ \ \nabla \overline{e} = \theta_2 e + \overline{\theta_1
}\overline{e} ,$$
thus the matrices of connection and curvature for $\nabla $ become:
$$ \hat{\theta } = \begin{bmatrix} \theta_1 & \overline{\theta_2 }\, \\ \theta_2 & \overline{\theta_1 }\,  \end{bmatrix}\! , \ \  \  \hat{\Theta } = \begin{bmatrix} \Theta_1 & \overline{\Theta}_2  \\ \Theta_2 & \overline{\Theta}_1   \end{bmatrix}, $$
 where
\begin{eqnarray*}
\Theta_1 & = & d\theta_1 -\theta_1 \wedge \theta_1 -\overline{\theta_2} \wedge \theta_2, \\
\Theta_2 & = & d\theta_2 - \theta_2 \wedge \theta_1 - \overline{\theta_1 } \wedge \theta_2,\\
d\varphi & = & - \ ^t\! \theta_1 \wedge \varphi - \ ^t\! \theta_2
\wedge \overline{\varphi } .
\end{eqnarray*}
As $e$ is unitary, both $\theta_2 $ and $\Theta_2$ are skew-symmetric, while $\theta$, $\theta_1$, $\theta^b$, or $\Theta$, $\Theta_1$, $\Theta^b$ are all skew-Hermitian. Consider the $(2,1)$-tensor $\gamma =\frac{1}{2}(\nabla^b -\nabla^c)$ introduced in \cite{YZ}. Its representation under the frame $e$ is a matrix of $1$-forms, which, by an abuse of notation, we will also denote by $\gamma$. Then it holds from \cite[Lemma 2]{WYZ} that
\begin{equation}\label{tht1}
\gamma = \theta_1 - \theta .
\end{equation}
Denote the decomposition of $\gamma$ into $(1,0)$ and $(0,1)$ parts by $\gamma = \gamma ' + \gamma ''$.
As observed in \cite{YZ}, when $e$ is unitary, $\gamma $ and $\theta_2$ take the following simple forms
\begin{equation}
(\theta_2)_{ij} = \sum_{k=1}^n \overline{T^k_{ij}} \varphi_k, \ \ \ \ \gamma_{ij} = \sum_{k=1}^n ( T_{ik}^j \varphi_k - \overline{T^i_{jk}} \overline{\varphi}_k ), \label{gm+tht2}
\end{equation}
while for general frames the above formulae will have the matrix $(g_{i\overline{j}})=(\langle e_i,\overline{e}_j \rangle )$ and its inverse involved. By (\ref{gm+tht2}), we get the expression of the components of the torsion $T^b$ under the unitary frame $e$
\begin{equation} \label{Tb}
T^b(e_i,e_j) = -2 \sum_{k=1}^n T^k_{ij}e_k, \ \ \ T^b(e_i, \overline{e}_j) = 2\sum_{k=1}^n \big( T^j_{ik}\overline{e}_k - \overline{T^i_{jk}} e_k \big).
\end{equation}
As usual, the curvature tensor $R^D$ of a linear connection $D$ on a Hermitian manifold $(M,g)$ is defined by
$$ R^D(x,y,z,w) = g( R^D_{xy}z, \ w ) = g(D_xD_yz-D_yD_xz - D_{[x,y]}z, \ w ), $$
where $x,y,z,w$ are tangent vectors in $M$. We will also write it as $R^D_{xyzw}$ for brevity. It is always skew-symmetric with respect to the first two positions, and is also skew-symmetric with respect to its last two positions if the connection is metric, namely, if $Dg=0$. Under a $(1,0)$-frame $e$, the components of the Chern, Bismut and Riemannian curvature tensors are given by
\begin{equation*}
R^c_{i\overline{j}k\overline{\ell}}  =  \sum_{p=1}^n \Theta_{kp}(e_i,
\overline{e}_j)g_{p\overline{\ell}}, \ \ \  \ \ R^b_{abk\overline{\ell}}  =  \sum_{p=1}^n \Theta^b_{kp}(e_a,
e_b)g_{p\overline{\ell}}, \ \ \ \ \ R_{abcd} = \sum_{f=1}^{2n} \hat{\Theta}_{cf}(e_a,e_b)g_{fd},
\end{equation*}
where $i,j,k,\ell,p$ range from $1$ to $n$, while $a,b,c,d,f$ range from $1$ to $2n$ with
$e_{n+i}=\overline{e}_i$, and $g_{ab}=g(e_a,e_b)$. For convenience, we will use a unitary frame $e$ throughout, where the curvature formulae above have simplified form.

For the remaining  part of this section, we will assume that $g$ is Bismut K\"ahler-like, or {\em BKL} for short, which means that the curvature tensor $R^b$ of $\nabla^b$ satisfies the symmetry condition
\[R^b_{ijk\overline{\ell}}=0\quad \text{and}\quad R^b_{i\overline{j}k\overline{\ell}}= R^b_{k\overline{j}i\overline{\ell}},\]
for any $1\leq i,j,k,\ell\leq n$ under any unitary frame $e$, or equivalently, as in \cite[Lemma 4]{ZZ},
\[ ^{t}\!\varphi \wedge \Theta^b=0.\]

We know from our foregoing work \cite{ZZ} and \cite{YZZ} that the {\em BKL} condition is equivalent to the Chern torsion tensor (or equivalently the torsion of any Gauduchon connection) being $\nabla^b$-parallel, plus the pluriclosedness of the metric. When we choose a local unitary frame $e$ and its dual coframe $\varphi$, with notations introduced above, under the {\em BKL} assumption, we have from \cite{ZZ} the equalities $T^j_{ik,\,\ell} = T^j_{ik, \,\overline{\ell}} =0$, $\,\sum_q \eta_q T^q_{ik}=0$, and
\begin{equation}
 \sum_q \{ T^q_{ik} \overline{ T^q_{j\ell} } + T^j_{iq} \overline{ T^k_{\ell q} } + T^{\ell}_{kq} \overline{ T^i_{jq} } - T^{\ell}_{iq} \overline{ T^k_{jq} } - T^j_{kq} \overline{ T^i_{\ell q} } \} =0 \label{eq:P=0}
 \end{equation}
 for any indices $i$, $j$, $k$, $\ell$. Here the indices after comma stand for the covariant derivatives with respect to $\nabla^b$, while $\eta$ stands for the Gauduchon torsion $1$-form $\eta =\sum_k \eta_k \varphi_k $, where
$ \eta_k = \sum_{i} T^i_{ik}$.
Also, {\em BKL} manifolds always satisfy $|T|^2=2|\eta|^2$ and $B=\phi + \phi^{\ast}$ by \cite{ZZ}, where
$$ A_{i\overline{j}} = \sum_{r,s} T^r_{is} \overline{T^r_{js}} , \
 \ \  B_{i\overline{j}} = \sum_{r,s} T^j_{rs} \overline{T^i_{rs}}, \ \ \ \phi_i^j = \sum_q T^j_{iq} \overline{\eta}_q  , $$
and $X_{\!\eta} = \sum_k \overline{\eta}_k e_k$ is a globally defined holomorphic vector field on $M^n$ which is parallel under $\nabla^b$ as shown in \cite[Theorem 2]{YZZ}. It has constant norm $|X_{\!\eta}| = |\eta| = \lambda \geq 0$. If we let $i=j$ and $k=\ell$ in (\ref{eq:P=0}), we get
\begin{equation} \label{eq:Psum=0}
 \sum_q \{ | T^q_{ik} |^2 + T^i_{iq} \overline{ T^k_{k q} } + T^{k}_{kq} \overline{ T^i_{iq} } - |T^{k}_{iq}|^2 - |T^i_{kq}|^2 \} =0
\end{equation}
for any $i$, $k$. Note that since $T$ is parallel under $\nabla^b$, so are $\eta$, $\phi$, $A$, and $B$. Also, we may regard $\phi$, $A$, or $B$ as (complex) smooth endomorphisms on $T^{1,0}M$, by sending $e_i$ to
 $$ \phi (e_i) = \sum_j \phi_i^j e_j, \ \ \ A(e_i) = \sum_j A_{i\overline{j}} \,e_j, \ \ \ B(e_i) = \sum_j B_{i\overline{j}} \,e_j, $$
respectively. The latter two are self-adjoint with respect to the Hermitian metric $g$ and are positive semi-definite. By an abuse of notation, we will denote these endomorphisms by the same letter. Since they are $\nabla^b$-parallel, their eigenvalues are global constants and their eigenspaces are all parallel under $\nabla^b$.

Now let us assume that $g$ is not K\"ahler, so $T\neq 0$ hence $\eta \neq 0$. Let $\mbox{ker}B$ be the zero eigenspace of $B$ in $T^{1,0}M$. That is, $X=\sum_i X_i e_i$ is in $\mbox{ker}B$ if and only if
$ B(X) = \sum_{i,j} X_i B_{i\overline{j}} e_j =0$, or equivalently, $\sum_i \overline{X}_i T^i_{rs}=0$ for any $r,s$. Since $\sum_q\eta_q T^q_{ik}=0$ for any $i,k$, we know that $X_{\!\eta}\in \mbox{ker}B$, hence $\mbox{ker}B \neq 0$. For any $X\in \mbox{ker}B$, let us denote by $P_X$ the endomorphism on $T^{1,0}M$ given by
$$ P_X (e_i) = \sum_j T^j_{i X}e_j = \sum_{j,k} T^j_{ik} X_k e_j .$$

\begin{lemma}\label{lemma1}
Let $(M^n,g)$ be a BKL manifold. For any $X,Y \in \mbox{ker}B$, the endomorphisms $P_X$ and $P_Y$ satisfy
$$ P_X P_Y^{\ast} = P_Y^{\ast} P_X,$$
where $P_Y^{\ast}$ denotes the adjoint (or conjugate transpose) of $P_Y$.
\end{lemma}

\begin{proof}
Write $X=\sum_k X_k e_k$ and $Y=\sum_{\ell} Y_{\ell} e_{\ell}$ under a local unitary frame $e$. We have $\sum_k \overline{X}_k T^k_{\ast \ast } =0$ and $\sum_{\ell} \overline{Y}_{\ell} T^{\ell}_{\ast \ast}=0$. Multiply $X_k \overline{Y}_{\ell}$ on both sides of (\ref{eq:P=0}) and sum up $k,\ell$, which implies
$$ \sum_q \{ T^q_{iX} \overline{ T^q_{jY} } - T^j_{qX} \overline{ T^i_{qY} } \} =0 $$
for any $i$, $j$. That is, $P_X P_Y^{\ast} = P_Y^{\ast} P_X $. This completes the proof of the lemma.
\end{proof}

In particular, if we let $X=Y$ in the above, we see that each $P_X$ is normal. Lemma \ref{lemma1} guarantees that we can simultaneously diagonalize $P_X$ for all $X\in \mbox{ker}B$. More precisely, let us first choose our unitary frame $e$ so that $X_{\!\eta} = \lambda e_n$. Note that $e_n$ is globally defined and $\lambda =|\eta|$ is a positive constant. Since $P_{e_n}=\frac{1}{\lambda} \phi $ is normal, we may take a unitary change of $\{ e_1, \ldots , e_{n-1}\}$ so that $P_{e_n}$ is diagonal, namely
\begin{equation}
 T^j_{in} = \delta_{ij} a_i, \label{eq:a}
 \end{equation}
where $a_1 a_2 \cdots a_r \neq 0$ and $a_{r+1}=\cdots = a_n=0$ for some $1\leq r\leq n-1$. Since $\phi$ is $\nabla^b$-parallel and $\phi_i^j=\lambda a_i\delta_{ij}$, these $a_i$'s are global constants and this $r$ here is the rank of $\phi$. It follows that $a_1 +\cdots +a_r=\lambda$. By $B=\phi+\phi^{\ast}$, we obtain that $\mbox{ker}\,\phi = \mbox{span}\{ e_{r+1}, \ldots , e_{n}\} \subseteq \mbox{ker}B$. On the other hand, we claim that $B_{i\bar{i}}>0$ for any $1\leq i\leq r$. Assume the contrary, then by the definition of the $B$ tensor we know that $T^i_{\ast \ast }=0$ for some $i$. Due to the unitary frame chosen above, we get $a_i=T^i_{in}=0$, a contradiction. So $B$ is positive definite on $\mbox{span}\{ e_{1}, \ldots , e_{r}\}$ thus $\mbox{ker}\,\phi = \mbox{ker}B$, and $r$ is also equal to the $B$-rank $r_{\! {\tiny B}}$ of $g$. In the rest of this paper, we will always use the letter $r$ to denote the $B$-rank of a {\em BKL} manifold $(M^n,g)$.

By Lemma \ref{lemma1}, for each $r+1\leq \alpha , \beta \leq n$, it yields that $P_{e_{\alpha}} P_{e_{\beta}}^{\ast} = P_{e_{\beta}}^{\ast} P_{e_{\alpha}} $, which indicates that we can choose a unitary frame $e$ so that all these $P_{e_{\alpha}}$ are also diagonal,
\begin{equation}
 T^j_{i\alpha} = \delta_{ij} b_{\alpha i} , \ \ \ r+1\leq \alpha \leq n-1, \label{eq:b}
 \end{equation}
if $r<n-1$. Note that, under this unitary frame $e$, it holds that
\begin{equation}
T^{\alpha }_{\ast \ast}=0, \ \ \ T^{\ast}_{\alpha \beta} =0,  \label{eq:alphabeta}
\end{equation}
for any $r+1\leq \alpha , \,\beta \leq n$. The first equality is established as $e_{\alpha} $ lies in the kernel of $B$, and the second equality holds when we let $i=\alpha$ and $k=\beta$ in (\ref{eq:Psum=0}). Hence, for any $r+1 \leq \alpha \leq n-1$, the possibly nonzero entries of $\{b_{\alpha i}\}_{i=1}^n$ are the first $r$ ones. We will see in Remark \ref{cst} that these $b_{\alpha i}$ can be local constants after the frame $e$ above is appropriately chosen. Since $\eta_1=\cdots =\eta_{n-1}=0$ by the choice of $e$, we have
\begin{equation}
 \sum_{i=1}^r b_{\alpha i} =0, \ \ \   r+1\leq \alpha \leq n-1. \label{eq:bsum}
 \end{equation}

\begin{lemma} \label{lemma2}
Suppose $(M^n,g)$ is a BKL manifold with $A>0$, then under the above frame and notation, the vector set $ \{ v_{r+1}, \ldots , v_{n-1},v_n\}$ is linearly independent, where $v_{\alpha} = (b_{\alpha 1}, \ldots , b_{\alpha r}) $, for $r+1 \leq \alpha \leq n-1$, and $v_n=(a_1,\ldots,a_r)$.
\end{lemma}

\begin{proof}
Assume the contrary, namely, there exist constants $c_{r+1}, \ldots , c_{n-1},c_n$, which are not all zeros, such that $c_{r+1}v_{r+1} + \cdots + c_{n-1}v_{n-1}+c_nv_n=0$. This means that $\sum_{\alpha =r+1}^{n} c_{\alpha} T^j_{i\alpha} =0$ for any $i$ and $j$. Then $X=\sum_{\alpha =r+1}^{n} c_{\alpha } e_{\alpha}\neq 0$ lies in the kernel of $A$, which is a contradiction. This completes the proof of the lemma.
\end{proof}

 Now we are ready to to prove Theorem \ref{thm1}.

\begin{proof}[{\bf Proof of Theorem \ref{thm1}}]

We will first show that $r \geq \frac{n}{2}$ when $A>0$. If $r=n-1$, we are done since $n\geq 2$. So we may assume that $r<n-1$.  From Lemma \ref{lemma2}, we know that the set of $n-r$ vectors $v_{r+1}, \ldots , v_{n-1},v_{n}$ in ${\mathbb C}^r$ are linearly independent,
which indicates that $n-r \leq r$, that is, $r\geq \frac{n}{2}$.

Next let us assume that $(M^n,g)$ is a complete {\em BKL} manifold. If it admits a (positive dimensional) K\"ahler de Rham factor, then it is clear that $A\not > 0$. For the converse, let us assume that $A\not > 0$ and we want to show that $M$ admits at least one K\"ahler de Rham factor.

Denote by $V=\mbox{ker}A$ the zero eigenspace of $A$. Since $A$ is parallel under $\nabla^b$, $V$ is a subbundle of $T^{1,0}M$ which is parallel under $\nabla^b$. Hence, around any given point $p\in M$, we may choose our local unitary frame $e$ near $p$ so that $V$ is spanned by $\{ e_1, \ldots , e_q\}$ and $\nabla^b e_i \in V$ for $1\leq i\leq q$. It follows that $T^{\ast}_{i\ast }=0$ as $e_i \in V$, for $1\leq i\leq q$. This implies $T^i_{\ast \ast }=0$ by (\ref{eq:Psum=0}). Then it yields from \eqref{tht1} and \eqref{gm+tht2} that, for $1\leq i\leq q$, the covariant derivative with resect to the Levi-Civita connection $\nabla$ of $e_i$ is
\begin{eqnarray*}
  \nabla e_i & = & \nabla^b e_i - \gamma_{ij} e_j + (\overline{\theta}_2)_{ij} \overline{e}_j \\
  & = & \nabla^b e_i  +  \sum_{j,k=1}^n \left\{ \big( - T^j_{ik} \varphi_k + \overline{ T^i_{jk}} \, \overline{\varphi}_k \big) e_j + T^k_{ij} \overline{\varphi}_k \overline{e}_j \right\} \\
  & = & \nabla^b e_i \ \in \ V.
\end{eqnarray*}
This means that $V$ is parallel with respect to $\nabla$, which implies that the universal covering space of $M^n$ admits a de Rham decomposition and the factor corresponding to $V$ is K\"ahler as the torsion tensor vanishes there. Then it is easy to see from above that $\mbox{ker}A \subseteq \mbox{ker}\phi \cap X_{\eta}^{\bot}$, and we already know $\mbox{ker}\,\phi = \mbox{ker}B$. This completes the proof of Theorem \ref{thm1}.
\end{proof}

This motivates the following three definitions.

\begin{definition}\label{irr}
Let $(M^n,g)$ be a BKL manifold. We say that it is {\bf full} if $A>0$. Clearly, a full BKL manifold is necessarily non-K\"ahler.
\end{definition}

\begin{definition}\label{cmpt}
For a non-K\"ahler BKL manifold $(M^n,g)$, we will call a local unitary frame $e$ with $e_n =\frac{1}{|\eta|} X_{\!\eta}$, satisfying (\ref{eq:a}) and (\ref{eq:b}), a {\bf $\phi $-compatible} frame. The integer $r$, which is the rank of $B$ or $\phi$, also denoted by $r_{B}$ sometimes, is called the {\bf rank} of $(M^n,g)$.
\end{definition}

Note that $r$ is the number of non-zero terms in $\{a_i\}_{i=1}^n$.

\begin{definition}\label{bbhat}
Under a $\phi $-compatible frame $e$ on a non-K\"ahler BKL manifold, let us denote by $b=(b_{\alpha i})$ the $s \times r$ matrix and by $\hat{b} = \begin{pmatrix} b \\ v_n \end{pmatrix}$ the $(s+1) \times r$ matrix, where $b_{\alpha i}$ and $v_n=(a_1,\ldots,a_r)$ are introduced above, and  $s=n-1-r$.
\end{definition}

\begin{remark}\label{Abhat}
Lemma \ref{lemma2} and Theorem \ref{thm1} imply that, for BKL manifolds, $A>0$ if and only if $\mbox{rank}\,(\hat{b})=n-r$.
\end{remark}

We will write
$$E=\mbox{span}\{ e_1, \ldots , e_r\}, \ \ \ N=\mbox{span}\{ e_{r+1}, \ldots , e_n\}, \ \ \ N'=\mbox{span}\{ e_{r+1}, \ldots , e_{n-1}\}. $$
Clearly, $E$ and $N'$ are parallel under $\nabla^b$ and $\mbox{ker}A \subseteq N'$, since $E$ is the direct sum of $\phi$-eigenspaces corresponding to non-zero eigenvalues, while $N'$ is the orthogonal complement of the $\nabla^b$-parallel direction $e_n$ in the zero eigenspace of $\phi$. After we let $\ell =n$ or $\ell =\alpha$ with $r+1\leq \alpha \leq n-1$ in (\ref{eq:P=0}), it follows that

\begin{lemma} \label{lemma3}
Let $(M^n,g)$ be a non-K\"ahler BKL manifold. Then under any $\phi$-compatible frame $e$, it holds that
\begin{eqnarray}
 (a_i+a_k-a_j) \, \overline{T^j_{ik}} & = & 0 \label{eq:aijk} \\
 (b_{\alpha i} + b_{\alpha k} - b_{\alpha j}) \,\overline{T^j_{ik} } & = & 0 \label{eq:bijk}
\end{eqnarray}
for any $i$, $j$, $k$, and any $r+1\leq \alpha \leq n-1$.
\end{lemma}

Note that the constants $a_i$ may not be all distinct. Let us denote by
$$E=E_1\oplus E_2\oplus \cdots \oplus E_p$$
the orthogonal decomposition of $E$ into $\phi$-eigenspaces with respect to distinct non-zero eigenvalues. In other words, by rearranging the order of $e_i$ if necessary, we may assume that
\begin{equation}
a_1 = \cdots = a_{n_1}, \ \ a_{n_1+1} = \cdots = a_{n_2}, \ \  \ldots , \ \ a_{n_{p-1}+1} = \cdots = a_{n_p}  \label{eq:aorder}
\end{equation}
where $p\geq 1$, $1\leq n_1 < n_2 < \cdots < n_p=r$, and $\{ a_{n_1}, a_{n_2}, \ldots , a_{n_p}\}$ are all distinct. For each $1\leq i\leq p$, the rank of $E_i$ is $m_i = n_i- n_{i-1}$, where $n_0$ is set to be $0$.

\begin{lemma}\label{lemma4}
Let $(M^n,g)$ be a non-K\"ahler BKL manifold. The endomorphism $A$ preserves the decomposition
$$ T^{1,0}M = E_1 \oplus E_2 \oplus \cdots \oplus E_p \oplus N,$$
\end{lemma}

\begin{proof}
Since $A_{i\overline{j}}=\sum_{k,\ell =1}^n T^k_{i\ell} \overline{ T^k_{j\ell }}$, it follows from (\ref{eq:aijk}) that $A_{i\overline{j}}=0$ unless $a_k=a_i+a_{\ell}$ and $a_k=a_j +a_{\ell}$ hold simultaneously. Hence, $A_{i\overline{j}}=0$ unless $a_i=a_j$. This means that the endomorphism $A$ will preserve the eigenspaces of $\phi$, hence the lemma is proved.
\end{proof}

It turns out that when $\phi$ is of maximal rank, namely, $r=n-1$, then there will be no multiplicity issue.

\begin{lemma}\label{lemma5}
Let $(M^n,g)$ be a BKL manifold with $n\geq 2$. If the $B$-rank $r=n-1$, namely, if $a_1, \ldots , a_{n-1}$ are all non-zero, then they are all distinct. In particular, $T^{1,0}M$ is the orthogonal sum of $\nabla^b$-parallel line bundles $L_i=\mbox{span}\{e_i\}$, $i=1, \ldots , n$.
\end{lemma}

\begin{proof}
Let $e$ be a local unitary frame that is $\phi$-compatible. If $a_i=a_k$ for some $1\leq i < k \leq n-1$, it follows from the equality (\ref{eq:aijk}) that $T^i_{kj}= T^k_{ij} = 0$ and $T^i_{ij}=T^k_{kj}=0$ for any $j\leq n-1$, as $a_j\neq 0$. Then the equality (\ref{eq:Psum=0})
implies
$$ \sum_q |T^q_{ik}|^2 + a_i \overline{a}_k + a_k \overline{a}_i = 0,$$
which is impossible since $a_i=a_k\neq 0$. This indicates that $a_1, \ldots , a_{n-1}$ are all distinct.
\end{proof}

Recall that a non-K\"ahler {\em BKL} manifold is said to have {\em degenerate torsion} if $T^{\ast}_{ik}=0$ for all $i,k<n$ where $e_n$ is parallel to $X_{\!\eta}$. By \cite{YZZ}, any  {\em BKL} manifold of dimension $n\leq 3$ always has degenerate torsion, where the $r = n-1$ condition does not imply that the manifold must be Bismut flat. In contrast, Theorem \ref{thm2} says that when $n\geq 4$,  the condition $r=n-1$ will actually force the {\em BKL} manifold to be Bismut flat and full.

\begin{proof}[{\bf Proof of Theorem \ref{thm2}}]
Under a $\phi$-compatible frame $e$, we have already seen from Lemma \ref{lemma5} that the nonzero constants  $a_1, \ldots , a_{n-1}$ are all distinct. Thus the connection matrix $\theta^b$ for $\nabla^b$ is diagonal, hence the curvature matrix $\Theta^b=d\theta^b$ is also diagonal. For $ i<n$, we will call $a_i$ an {\em isolated root,} if $T^i_{jk} = T^j_{ik}=0$ for any $j, k <n$. We will first show that there is actually no isolated $a_i$.

Now suppose that there is an isolated root $a_i$. Then the equality (\ref{eq:Psum=0}) implies $\mbox{Re} (a_i \overline{a}_k) =0$ for $k<n$ and $k\neq i$. This means that the vectors $(a_i, \overline{a}_i)$ and $(a_k, \overline{a}_k)$ are perpendicular in ${\mathbb C}^2$. The fact $n-1\geq 3$ implies that there can be at most one isolated root. Without loss of generality, we may assume that $a_{n-1}$ is the isolated root. Hence $(a_k, \overline{a}_k)$ are parallel to each other for all $1\leq k \leq n-2$, and thus we can write $a_k=t_kc$, where $\{t_k\}_{k=1}^{n-2}$ are non-zero distinct real constants and $c$ is a non-zero constant.

These real numbers $t_k$ cannot all have the same sign. If not, say they are all positive, we may assume that $0<t_1<t_2< \cdots < t_{n-2}$. Then Lemma \ref{lemma3} indicates that $T^1_{2q}=0$ for any $q$ since $a_1<a_2$, and $T^2_{1q}=0$ for any $q$ since $a_2-a_1=a_q$ is impossible for any $q>1$. Let $i=1$ and $k=2$ in (\ref{eq:Psum=0}), which yields
$$ 2t_1t_2 |c|^2 + \sum_q |T^q_{12}|^2 =0.$$
It is a contradiction as the left hand side has positive first term while the other terms are non-negative. Hence those $t_k$ must change sign.

Again without loss of generality, we may assume that those $t_k$ are in the form
$$ \cdots < t_1 < 0 < t_2 < \cdots ,$$
namely, $t_1$ is the largest negative number while $t_2$ is the smallest positive number in the set $\{ t_1, \ldots , t_{n-2}\}$. If $T^q_{12}\neq 0$ for some $q<n$, then $a_q=a_1+a_2$ by Lemma \ref{lemma3}. It is clear that $q < n-1$ here as $(a_{n-1},\overline{a}_{n-1})$ is perpendicular to $(a_{1},\overline{a}_{1})$ and $(a_{2},\overline{a}_{2})$. It follows that $t_q=t_1+t_2$. If $t_q<0$, then since $t_q=t_1+t_2>t_1$, it will contradict with the assumption that $t_1$ is the largest negative root. Similarly, if $t_q>0$, then it will also contradict with the assumption that $t_2$ is the smallest positive root. This shows that we must have $T^q_{12}=0$ for any $q$. Let $i=1$ and $k=2$ in (\ref{eq:Psum=0}), it now yields
$$ 2t_1t_2 |c|^2  - \sum_q \{ |T^1_{2q}|^2 + |T^2_{1q}|^2\} =0,$$
which is again a contradiction since the terms on the left are all non-positive, with the first one being negative. Therefore, there cannot be any isolated root amongst those $a_i$'s.

Then we will show that, for any $i<n$, $\Theta^b_{ii}=0$. This implies that the manifold will be Bismut flat. Now as $a_i$ is not an isolated root, there exists $j,k<n$ such that $T^j_{ik}$ or $T^i_{jk}$ is not zero. Without loss of generality, we assume that $T^j_{ik}\neq 0$. It is clear that $i\neq k$, as $T$ is skew-symmetric with respect to its two lower indices. It follows from Lemma \ref{lemma3} that $a_j=a_i+a_k$. In particular, $j$ can neither be equal to $i$ nor $k$ as that would imply that $a_k$ or $a_i$ is zero. Note that each $e_i$ is globally determined up to a function with unit norm, thus $|T^j_{ik}|$ is a global function. As the torsion $T$ is $\nabla^b$-parallel, $|T^j_{ik}|$ turns out to be a global constant. Let us rotate $e_i$ while fixing the others to make $T^j_{ik} = |T^j_{ik}|$. For any tangent vector $X$ , we have
$$\begin{aligned}
0 = T^j_{ik,X} &= X(T^j_{ik}) - T^j_{qk}\theta^b_{iq}(X) - T^j_{iq} \theta^b_{kq}(X) + T^q_{ik}\theta^b_{qj}(X) \\
               &= -T^j_{ik} \{ \theta^b_{ii}(X) + \theta^b_{kk}(X) - \theta^b_{jj}(X) \},
\end{aligned}$$
where we used the fact that $\theta^b$ is diagonal. This implies that $\theta^b_{ii} + \theta^b_{kk} = \theta^b_{jj}$, thus
\begin{equation}\label{eq:Thetaijk} \Theta^b_{ii} + \Theta^b_{kk}=\Theta^b_{jj}.
\end{equation}
On the other hand, since $\Theta^b$ is diagonal, the Bismut K\"ahler-like condition $\,^t\!\varphi \, \Theta^b =0$ implies that $\Theta^b_{qq} = \kappa_q \varphi_q \overline{\varphi}_q$ for each $q$, where $\kappa_q$ is a local real valued function. Note that $i$, $j$, $k$ are distinct. Hence, the equation (\ref{eq:Thetaijk}) implies that $\Theta^b_{ii}=0$.

When $r=n-1$, $\mbox{ker}B \cap X_{\eta}^{\bot} = 0$, so by Theorem \ref{thm1} we know that $\mbox{ker}A=0$, namely,  $A>0$ so $M$ is full. This completes the proof of Theorem \ref{thm2}.
\end{proof}

Note that all Bismut flat manifolds were classified in \cite{WYZ}. So for {\em BKL} manifolds with $n\geq 4$, the emphasis will be laid on the cases  $r<n-1$. Under a $\phi$-compatible frame $e$, the connection matrix $\theta^b$ is always block diagonal with respect to the decomposition
$$ T^{1,0}M = E \oplus N = E_1 \oplus \cdots \oplus E_p \oplus N $$
into the eigenspaces of $\phi$. This results from the equation
\begin{equation}
0 = T^j_{ik, X} = X(T^j_{ik}) - T^j_{qk}\theta^b_{iq}(X) -  T^j_{iq}\theta^b_{kq}(X) +  T^q_{ik}\theta^b_{qj}(X), \label{eq:Tderivative}
\end{equation}
where $X$ is any tangent vector. Let $k=n$, we get $(a_i-a_j)\theta^b_{ij} = 0$, therefore $\theta^b_{ij}=0$ whenever $a_i\neq a_j$. When $r<n-1$, those non-zero $a_j$ may no longer be all distinct. Nonetheless the connection matrix $\theta^b$ is still diagonal over the $E$ part.

\begin{lemma}\label{lemma6}
Let $(M^n,g)$ be a non-K\"ahler BKL manifold. If the $B$-rank $r<n-1$, then for any given $\phi$-compatible frame $e$, the connection matrix $\theta^b$ is block diagonal with respect to the eigenspace decomposition of $\phi$, and $\theta^b$ is diagonal over the $E$ part, namely, $\theta^b_{ij}=0$ for any $1\leq i<j\leq r$.
\end{lemma}

\begin{proof}
It suffices to show the diagonal part when $r<n-1$. We will first show that, if $a_i=a_k$ for some $1\leq i<k\leq r$, there exists $\alpha $ with $r+1\leq \alpha \leq n-1$ such that $b_{\alpha i} \neq b_{\alpha k}$.

The strategy of the proof of Theorem \ref{thm2} could  also be applied here. It follows from (\ref{eq:aijk}) in Lemma \ref{lemma3} that $T^i_{kj}=T^k_{ij}=0$ for any $j\leq r$, since $a_i=a_k$ and $a_j\neq 0$. Then the equality (\ref{eq:Psum=0}) implies that
$$ \sum_q |T^q_{ik}|^2 + 2\mbox{Re} \{ \,a_i \overline{a}_k + \!\sum_{\alpha =r+1}^{n-1} \!b_{\alpha i} \overline{b_{\alpha k}} \,\}  = 0.$$
Hence, if $b_{\alpha i}=b_{\alpha k} $ for all $\alpha$, the above is clearly a contradiction, which indicates that $b_{\alpha i} \neq b_{\alpha k}$ for some $\alpha$.

Then let us take $k=\alpha$ and $i,j\leq r$ in (\ref{eq:Tderivative}), which implies
$$ 0 = X(b_{\alpha i} \delta_{ij}) + (b_{\alpha i} - b_{\alpha j}) \theta^b_{ij}(X)
- \sum_{\beta=r+1}^{n-1} T^j_{i\beta} \theta^b_{\alpha \beta }(X), $$
where in the last term above we used the fact that $\theta^b_{\alpha n} =0$, and $\theta^b_{\alpha q} =0$ for any $q\leq r$, as $a_q\neq a_{\alpha}$. For $i\neq j$, the above equality gives us $(b_{\alpha i} - b_{\alpha j}) \theta^b_{ij} =0$, thus $\theta^b_{ij}=0$ holds. This completes the proof of the lemma.
\end{proof}

For the $N$ part, we can actually always choose $\phi$-compatible frame $e$ appropriately when $A>0$, so that the connection matrix $\theta^b$ is zero over the $N$ part, hence is diagonal in the case $r<n-1$.

Let $(M^n,g)$ be a {\em BKL} manifold with $A>0$ and $e$ be a local $\phi$-compatible frame. Write $r=r_{\!{\tiny B}}$ and $s=n-1-r$. Note that the $s\times r$ matrix $b=(b_{\alpha i})$ is introduced in Definition \ref{bbhat}. Let us denote by $\theta^b_1=(\theta^b_{\alpha \beta})$ the $s\times s$ matrix. Denote by $P$ the left $s\times s$ corner of $b$. It follows from Lemma \ref{lemma2} that $\mbox{rank}\, (b) =s$. Then, by rearranging the order of $\{ e_1, \ldots , e_r\}$ if necessary, we may assume that $P$ is non-singular. Since $\theta^b$ is block diagonal for $E\oplus N$ and diagonal over $E$, we have  for any  $1\leq i \leq r$, $r+1 \leq \alpha \leq n-1$ that
$$ 0 = T^i_{i\alpha , X} = X(T^i_{i\alpha}) - \sum_q \{ T^i_{q\alpha} \theta^b_{iq}(X)+T^i_{iq} \theta^b_{\alpha q}(X) - T^q_{i\alpha} \theta^b_{qi}(X) \} .$$
Hence $d T^i_{i\alpha} = \sum_{\beta=r+1}^{n-1} \theta^b_{\alpha \beta} T^i_{i\beta}$, or $db = \theta^b_1 b$. This implies that $ dP = \theta^b_1 P$, therefore we have $\theta^b_1 = dP P^{-1}$ and
$$ d\theta^b_1 - \theta^b_1 \theta^b_1 = d(dP P^{-1}) - dP P^{-1}dP P^{-1} = 0.$$
That is, $\Theta^b_{\alpha \beta}=0$ for any $r+1 \leq  \alpha , \beta \leq n-1$. Therefore, the Bismut curvature is flat over $N$. This means that we can always choose a $\phi$-compatible frame $e$ so that $\nabla^b e_{\alpha}=0$ for each $r+1 \leq \alpha \leq n-1$.

\begin{definition}
For a BKL manifold $(M^n,g)$, a $\phi$-compatible local frame $e$ is called {\bf strictly $\phi$-compatible,} if $\nabla^b e_{\alpha}=0$ for each $r+1 \leq \alpha \leq n-1$, where $r$ is the $B$-rank.
\end{definition}

\begin{remark}\label{cst}
Under a strictly $\phi$-compatible frame $e$, each $b_{\alpha i}$ is now a local constant, and each $e_{\alpha}$ is a local holomorphic vector field, as $T^{\alpha}_{\ast \ast}=0$ implies $\nabla^c_{\overline{e}_q} e_{\alpha} = \nabla^b_{\overline{e}_q} e_{\alpha}
- 2 \gamma_{\overline{e}_q} e_{\alpha} =0$. Also, each $e_{\alpha}$ can be globally defined if $M^n$ is simply-connected. At this time, for each $\alpha$, $b_{\alpha 1},\cdots,b_{\alpha r}$ turn out to be global constants, as they are actually the eigenvalues of the endomorphism $T^c(\cdot \,,e_{\alpha})$ of the globally defined distribution $E$, where $E=\mbox{span}\{e_1,\cdots,e_r\}$. While $e_n$ is actually always a globally defined holomorphic $\nabla^b$-parallel vector field as $X_{\eta}=\lambda e_n$.
\end{remark}

Summarize what we discussed above, we have obtained the following

\begin{lemma} \label{s-cmpt}
Let $(M^n,g)$ be a BKL manifold with $A>0$. Then given any $p\in M$, there exists a local unitary frame $e$ in a neighborhood of $p$ such that $e$ is strictly $\phi$-compatible.
\end{lemma}

Now we are ready to prove Theorem \ref{thm3} stated in the introduction.

\begin{proof}[{\bf Proof of Theorem \ref{thm3}}]
Let $(M^n,g)$ be a {\em BKL} manifold without any K\"ahler de Rham factor of dimension bigger than one. Assume that it is Calabi-Yau with torsion ({\em CYT}), namely assume that $\mbox{tr} (\Theta^b )=0$. We want to conclude that $\Theta^b=0$. Note that for any one dimensional K\"ahler de Rham factor, the {\em CYT} assumption forces it to be flat, so we may ignore them and assume that $M$ admits no K\"ahler de Rham factor, or equivalently, $A>0$ ($M$ is full). By Lemma \ref{s-cmpt}, locally we always have strictly $\phi$-compatible frame $e$, under which the connection matrix $\theta^b$ is diagonal over the $E=\mbox{span}\{ e_1, \ldots , e_r\}$ block and zero elsewhere, and thus $\Theta^b$ is also diagonal. It follows from the {\em BKL} condition that $\,^t\!\varphi \, \Theta^b =0$, which implies that $\Theta^b_{ii} = \kappa_i \varphi_i \overline{\varphi}_i$ for each $1\leq i\leq r$, while $\Theta^b_{\alpha \alpha} =0$ for any $r<\alpha \leq n$. Therefore, it holds that
$$ 0= \mbox{tr} (\Theta^b) = \sum_{i=1}^r \kappa_i \varphi_i \overline{\varphi}_i, $$
which implies that each $\kappa_i=0$, hence $\Theta^b=0$. This completes the proof of Theorem \ref{thm3}.
\end{proof}

\vs

\section{Proof of Theorem \ref{thm4}.}\label{BKL4D}
We are ready to prove Theorem \ref{thm4} in this section. For the first part of this theorem, let $(M^n,g)$ be a complete full {\em BKL} manifold and $r=r_{\!{\tiny B}}=\frac{n}{2}$, which implies $r=s+1$ where $s=n-1-r$, so the matrix $\hat{b}$ is an invertible $r \times r$ matrix by Lemma \ref{lemma2}. For convenience, we may set $b_{ni} = a_i$ so that the entries of $\hat{b}$ can be unified as $(b_{\alpha i})$, where $r+1 \leq \alpha \leq n, 1 \leq i \leq r$. Let us choose a strictly $\phi$-compatible frame $e$, and assume that $a_1a_2\cdots a_r\neq 0$ and $a_{r+1}=\cdots=a_n=0$. It follows from Remark \ref{cst} that $e_{r+1},\cdots,e_{n}$ are globally defined, holomorphic, $\nabla^b$-parallel vector fields on the universal covering $\widetilde{M}$ of $M^n$, and $b_{\alpha i}$ are global constants of $\widetilde{M}$.

It follows that $T_{ik}^j=0$ for $1 \leq i,j,k \leq r$ in this case. If not, that is, $T_{ik}^j \neq 0$ for some $i,j,k \leq r$, Lemma \ref{lemma3} indicates that $b_{\alpha j} = b_{\alpha i} + b_{\alpha k}$ for any $r+1 \leq \alpha \leq n$. So the sum of the $i$-th and $k$-th columns equal to the $j$-th column, contradicting with the fact that $\hat{b}$ is an invertible matrix. Now let us compute the covariant derivatives with respect to the Levi-Civita connection $\nabla$ on $\widetilde{M}$.  For $1 \leq i \leq r$ and $r+1 \leq \alpha \leq n$, we have
\begin{eqnarray*}
\nabla e_i & = & \nabla^b e_i - \gamma (e_i) + \overline{\theta_2} (e_i) \\
& = & \theta^b_{ii} e_i + \sum_{q,k} \{  \overline{T^i_{qk}} \overline{\varphi}_k - T^q_{ik} \varphi_k \} e_q + \sum_{q,k}  T^k_{iq} \overline{\varphi}_k \overline{e}_q\\
& = & \left( \theta_{ii}^b + \sum_{\alpha=r+1}^n \overline{b}_{\alpha i} \overline{\varphi}_{\alpha} - b_{\alpha i} \varphi_{\alpha} \right)e_i
+ \overline{\varphi}_i \left( \sum_{\alpha=r+1}^n b_{\alpha i} \overline{e}_{\alpha} - \overline{b}_{\alpha i} e_{\alpha}\right),\\
\end{eqnarray*}
and similarly,
\begin{eqnarray*}
\nabla e_{\alpha} & = &  \nabla^b e_{\alpha} - \gamma (e_{\alpha}) + \overline{\theta_2} (e_{\alpha}) \\
& = & -\sum_{q,k} T_{\alpha k}^q \varphi_k e_q + \sum_{q,k} T_{\alpha q}^k \overline{\varphi}_k \overline{e}_q \\
& = & \sum_{k=1}^r b_{\alpha k}(\varphi_k e_k - \overline{\varphi}_k \overline{e}_k).
\end{eqnarray*}
For $1 \leq i \leq r$, let us write
\[X_i = - \sum_{r+1 \leq \alpha \leq n} \overline{b}_{\alpha i} e_{\alpha}, \]
then we have
\[\nabla X_i = - \sum_{\begin{subarray}{c} 1 \leq k \leq r \\ r+1 \leq \alpha \leq n \end{subarray}}
\overline{b}_{\alpha i} b_{\alpha k}(\varphi_k e_k - \overline{\varphi}_k \overline{e}_k).\]
As the matrix $\hat{b}$ is invertible, the globally defined vector fields $X_1,X_2,\cdots,X_r$ are also a basis of $N=\mbox{span}\{ e_{r+1},\cdots,e_n \}$. From $d\varphi = - ^t\!\theta^c \varphi + \tau = - ^t\!\theta^b \varphi + 2 ^t\!\gamma \varphi + \tau$, it yields that, for $1 \leq i \leq r$, $r+1 \leq \alpha \leq n$,
\begin{align}
d \varphi_i &= - \theta^b_{ii} \varphi_i + 2 \sum_k \gamma_{k i} \varphi_k + \tau^i  \notag\\
&=  - \theta^b_{ii} \varphi_i + 2 \sum_{\alpha =r+1}^n \varphi_i (\overline{b}_{\alpha i} \overline{\varphi}_\alpha - b_{\alpha i} \varphi_\alpha),\label{dvarphi} \\
d \varphi_\alpha &= 2 \sum_k \gamma_{k \alpha} \varphi_k + \tau^\alpha \notag \\
&= - 2 \sum_{k=1}^r \overline{b}_{\alpha k} \varphi_k \overline{\varphi}_k,\label{dvarphii}
\end{align}
which implies, for each $1\leq k \leq r$, $\varphi_k \overline{\varphi}_k$ is globally defined as $d \varphi_\alpha$ is, and thus $e_k$ is globally determined up to a function of unit norm. This indicates that, for $1 \leq i \leq r$, the distribution $F_i = \mbox{span}\{e_i,X_i\}$ is globally defined on $\widetilde{M}$. Since $X_i$ is holomorphic and
$$ \nabla^c_{\overline{Z}} e_i = \nabla^b_{\overline{Z}} e_i - 2\gamma_{\overline{Z}} e_i = (\theta^b_{ii}(\overline{Z}) + 2\overline{T^i_{iZ}}) e_i + 2 \overline{ \varphi_i(Z)}X_i,$$
we see that $ \nabla^c_{\overline{Z}} F_i \subseteq F_i$. Hence, $F_i$ is a holomorphic distribution, and we get a global holomorphic decomposition $T^{1,0}\widetilde{M} =  \oplus_{i=1}^r F_i$. It follows from the $\nabla$-covariant derivative formulae for $e_i$ and $X_i$ above that
$$ \nabla e_i \equiv 0, \qquad \nabla X_i \equiv - \sum_{\begin{subarray}{c} 1 \leq k \leq r, k\neq i \\ r+1 \leq \alpha \leq n \end{subarray}}
\overline{b}_{\alpha i} b_{\alpha k}(\varphi_k e_k - \overline{\varphi}_k \overline{e}_k) \qquad \mbox{modulo} \ F_i \oplus \overline{F}_i.$$
Hence $F_i \oplus \overline{F}_i$ is closed under the Lie bracket and thus a foliation. Therefore, $\widetilde{M}$ decomposes into a product of $r$ complex surfaces $N_1,\cdots,N_r$, where $N_i$ is generated by $F_i \oplus \overline{F}_i$, and the product is in general not a metric product.

Note that the K\"ahler form of the metric $g$ is
\[\omega_g = \sqrt{-1}\sum_{i=1}^r( \varphi_i \overline{\varphi}_i + \varphi_{r+i} \overline{\varphi}_{r+i})\]
and $\varphi_i \overline{\varphi}_i$ is globally defined for each $1 \leq i \leq r$, while, $\sqrt{-1}( \varphi_i \overline{\varphi}_i + \varphi_{r+i} \overline{\varphi}_{r+i})$ might not define a Hermitian metric on each $N_i$. So we let $\{\psi_{r+i}\}_{i=1}^r$ be the dual basis of $\{X_i\}_{i=1}^r$ and it is clear that, for $1 \leq j \leq r$,
\[\psi_{r+j} = - \sum_{r+1 \leq \beta \leq n} \overline{c}_{j \beta} \varphi_\beta,\]
where $(c_{j\beta})$ is the inverse matrix of $\hat{b}=(b_{\alpha i})$. Then the globally defined form $\omega_i = \sqrt{-1}(\varphi_i \overline{\varphi}_i + \psi_{r+i} \overline{\psi}_{r+i})$ defines a natural Hermitian metric on $N_i$. It follows that
\[\omega_h = \sqrt{-1} \sum_{i=1}^r(\varphi_i \overline{\varphi}_i + \psi_{r+i} \overline{\psi}_{r+i}) \]
defines the product metric on $\widetilde{M} = N_1 \times \cdots \times N_r$. We will see that $(N_i,\omega_i)$ is actually a non-K\"ahler {\em BKL} surface for each $i$, and thus $(\widetilde{M}, h)$ is also {\em BKL}.

Let $1 \leq i<k \leq r$ in (\ref{eq:Psum=0}) and we get
\begin{equation}\label{Re=0}
\sum_{r+1 \leq \alpha \leq n}\mbox{Re} (b_{\alpha i}\overline{b}_{\alpha k}) = 0.
\end{equation}
It follows from the equalities \eqref{dvarphi}, \eqref{dvarphii} and \eqref{Re=0} that, for $1 \leq i \leq r$,
\[ d \psi_{r+i} = 2 \varphi_i \overline{\varphi}_i ,\]
\[d \varphi_i = -\big( \theta_{ii}^b + \sum_{r+1 \leq \alpha \leq n} 2|b_{\alpha i}|^2 (\psi_{r+i} - \overline{\psi}_{r+i})
+ \sum_{\begin{subarray}{c} 1 \leq j \leq r, j\neq i \\ r+1 \leq \alpha \leq n \end{subarray}}
2b_{\alpha i } \overline{b}_{\alpha j}(\psi_{r+j} + \overline{\psi}_{r+j}) \big) \varphi_i. \]
Denote by
\[\xi_i =  \theta_{ii}^b + \sum_{r+1 \leq \alpha \leq n} 2|b_{\alpha i}|^2 (\psi_{r+i} - \overline{\psi}_{r+i})
+ \sum_{\begin{subarray}{c} 1 \leq j \leq r, j\neq i \\ r+1 \leq \alpha \leq n \end{subarray}}
2b_{\alpha i } \overline{b}_{\alpha j}(\psi_{r+j} + \overline{\psi}_{r+j}),\]
and we get, under the unitary frame $(e_i,X_i)$ on $(N_i,\omega_i)$,
$$ d \!\begin{bmatrix} \varphi_i \\ \psi_{r+i} \end{bmatrix} \! =
\! \begin{bmatrix} -\xi_i & 2\varphi_i \\ -2\overline{\varphi}_i & 0\\ \end{bmatrix}\!
\begin{bmatrix} \varphi_i \\ \psi_{r+i} \end{bmatrix} +
\begin{bmatrix} -2\varphi_i \psi_{r+i} \\  0 \end{bmatrix}\!. $$
It is clear that the square matrix is skew-Hermitian, and the last column is of type $(2,0)$, which yields the Chern connection matrix $\tilde{\theta}^c$ and the Chern torsion $\tilde{\tau}$ under $(e_i,X_i)$
\[\tilde{\theta}^c = \! \begin{bmatrix} \xi_i & 2\overline{\varphi}_i \\ -2 \varphi_i & 0 \end{bmatrix}\!, \quad
\tilde{\tau} = \! \begin{bmatrix} -2\varphi_i \psi_{r+i} \\ 0 \end{bmatrix}\!.\]
It indicates that the torsion components are
$$ \widetilde{T}^1_{12} = -1,\quad \widetilde{T}^2_{12} = 0. $$
Then the tensor $\gamma$ amounts to
$$ \tilde{\gamma} = \!\begin{bmatrix} -\psi_{r+i} + \overline{\psi}_{r+i} & - \overline{\varphi}_i \\ \varphi_i & 0 \end{bmatrix}\!,$$
and the Bismut connection matrix $\tilde{\theta}^b$ becomes
$$ \tilde{\theta}^b = \tilde{\theta} + 2 \tilde{\gamma} =
\!\begin{bmatrix} \xi_i -2\psi_{r+i} + 2\overline{\psi}_{r+i} & 0  \\ 0 & 0 \end{bmatrix}\!.$$
It follows from the proof of Theorem \ref{thm3} that
$$ \widetilde{\Theta}^b =\! \begin{bmatrix} \Theta_{ii}^b + 8(\sum\limits_{\alpha = r+1}^n |b_{\alpha i}|^2 -1) \varphi_i \overline{\varphi}_i & 0 \\ 0 & 0 \end{bmatrix} \!
= \!\begin{bmatrix} \big( \kappa_i + 8(\sum\limits_{\alpha = r+1}^n |b_{\alpha i}|^2 -1) \big) \varphi_i \overline{\varphi}_i & 0 \\ 0 & 0 \end{bmatrix}\!.$$
Therefore, it is now clear that $\widetilde{\Theta}^b$ satisfies $ \, (\varphi_i, \psi_{r+i}) \widetilde{\Theta}^b =0$, which means that the complex surface $(N_i, \omega_i)$ is {\em BKL}.

It is not hard to see that, amongst all the constant linear changes of $\{ e_{r+1}, \cdots, e_{n}\}$ into the new vectors $\{ X_1,\cdots, X_r\}$, the one we did above is exactly those which makes $\omega_h$ remain to be pluriclosed. For this reason, we will call $h$ a {\em pluriclosed modification} of $g$, or $g$ a {\em pluriclosed product} of the {\em BKL} surfaces $N_1,\cdots,N_r$, twisted by the matrix $(-\overline{b}_{\alpha i})$. This motivates the following definition.

\begin{definition}\label{p-twisted}
Given $r$ non-K\"ahler {\em BKL} surfaces $(N_1,h_1),\cdots, (N_r,h_r)$ and an invertible $r \times r$ matrix $D=(d_{ij})$, satisfying $\sum_i d_{ij} \overline{d_{ik}} + d_{ik} \overline{d_{ij}} =0$ for $1 \leq j < k \leq r$, we can construct a Hermitian metric $g$ on the product manifold $\Pi_{i=1}^rN_i = N_1 \times \cdots \times N_r$ that is {\em BKL}.

Let $(e_1^i,e_2^i)$ be a $\phi$-compatible frame on $(N_i,h_i)$ with the dual frame denoted by $(\varphi_1^i, \varphi_2^i)$ for $1 \leq i \leq r$, where $e_2^i$, $\varphi_2^i$ and $\varphi_1^i \overline{\varphi_1^i}$ are clearly globally defined. Twist $\varphi_2^i$ by the matrix $D$ as
\[\psi_2^i = \sum_j d_{ij} \varphi_2^j.\]
Define
\[\omega_g = \sqrt{-1} \sum_i (\varphi_1^i \overline{\varphi_1^i} + \psi_2^i \overline{\psi_2^i}),\]
which induces a globally defined Hermitian metric $g$ on $\Pi_{i=1}^rN_i$, while, the product metric $h$ on $\Pi_{i=1}^rN_i$ is given by
\[ \omega_h = \sqrt{-1} \sum_i(\varphi_1^i \overline{\varphi_1^i} + \varphi_2^i \overline{\varphi_2^i}). \]
We will say the Hermitian metric $g$ is a {\em pluriclosed modification} of $h$, or a {\em pluriclosed twisted product} of the {\em BKL} surfaces
$N_1, \cdots , N_r$, with the Hermitian manifold $(\Pi_{i=1}^rN_i,g)$ denoted by $N_1 \times_{\!D} \cdots \times_{\!D} N_r$. The Hermitian metric $g$ is actually {\em BKL} as shown below, which justifies the name pluriclosed.
\end{definition}

\begin{proposition}
The pluriclosed twisted product $N_1 \times_{\!D} \cdots \times_{\!D} N_r$ defined above is BKL.
\end{proposition}

\begin{proof}
As $(e_1^i,e_2^i)$ is a $\phi$-compatible frame on the non-K\"ahler BKL surface $(N_i,h_i)$, with dual frame denoted by $(\varphi_1^i,\varphi_2^i)$, for $1 \leq i \leq r$.
It follows from $d\varphi = -^t\!\theta^b \varphi + 2 ^t\!\gamma \varphi + \tau$ that
\begin{align*}
d \varphi_1^i &= (- ^i\!\theta^b_{11} + 2 \lambda_i (\varphi^i_2 - \overline{\varphi}^i_2))\varphi^i_1, \\
d \varphi_2^i &= -2 \lambda_i \varphi_1^i \overline{\varphi}_1^i,
\end{align*}
where $\lambda_i>0$ is the only non-zero eigenvalue of $\frac{1}{\lambda_i}\phi_i$ on $(N_i,h_i)$ and $^i\!\theta^b_{11}$ comes from the Bismut connection matrix $^i\!\theta^b$ under $(e_1^i,e_2^i)$, satisfying $d \,^i\!\theta^b_{11}= \kappa_i \varphi_1^i \overline{\varphi}_1^i$ for some local real function $\kappa_i$. Note that
\[\psi_2^i = \sum_j d_{ij} \varphi_2^j,\]
where inverse matrix of $D$ is denoted by $C=(c_{ij})$. It implies
\begin{align*}
d \varphi_1^i &= (- ^i\!\theta^b_{11} + 2 \lambda_i \sum_j (c_{ij}\psi^j_2 - \overline{c}_{ij}\overline{\psi}^j_2))\varphi^i_1, \\
d \psi_2^i &= -2 \sum_j d_{ij} \lambda_j \varphi_1^j \overline{\varphi}_1^j,
\end{align*}
Then we get, under the unitary frame $(\varphi^i_1,\psi^i_2)$ of the Hermitian metric $g$, that
\[d\!\begin{bmatrix} \varphi^i_1 \\ \psi^i_2\end{bmatrix}\!=\!\begin{bmatrix} R & Q \\ -^t\!\overline{Q} & 0 \end{bmatrix}\!
\begin{bmatrix} \varphi^i_1 \\ \psi_2^i \end{bmatrix}\! + \!\begin{bmatrix} \beta_i \\ 0 \end{bmatrix}\!,\]
where
\begin{gather*}
R_{ij}= \Big( -\, ^i\!\theta_{11}^b +
2 \lambda_i \sum_k \big( c_{ik} \psi^k_2 - \overline{c_{ik} \psi^k_2} \big) \Big) \delta_{ij}, \\
Q_{i j} = -2 \overline{d}_{ji} \lambda_i \varphi_1^i ,\quad
\beta_i = 2 \lambda_i \sum_k \overline{d}_{ki} \varphi_1^i \psi_2^k .
\end{gather*}
It follows that the Chern connection matrix $\theta^c$ and Chern torsion $\tau$ are given by
\[\theta^c =\!\begin{bmatrix} \overline{R} & \overline{Q} \\ -^t\!Q & 0 \end{bmatrix}\!, \quad \tau^i = \beta_i,\quad \tau^{r+i}=0. \]
This indicates that the possibly non-zero Chern torsion components are
\[ T^i_{j\,r+k} = \lambda_i \overline{d}_{ki} \delta_{ij},\]
which implies the tensor $\gamma$ is given by
\begin{gather*}
\gamma_{ij} = \lambda_i \delta_{ij} \sum_k (\overline{d}_{ki} \psi_2^k - d_{ki} \overline{\psi}_2^k ),\\
\gamma_{i\,r+j} = \lambda_i d_{ji} \overline{\varphi}_1^i,\quad
\gamma_{r+i\,j} = -\overline{\gamma_{j\, r+i}}, \quad \gamma_{r+i\,r+j}=0.
\end{gather*}
Hence the Bismut connection matrix $\theta^b$ yields
\[ \theta^b = \theta^c + 2\gamma = \!\begin{bmatrix} S & 0 \\ 0 & 0 \end{bmatrix}\!,\]
where
\[S_{ij} = \Big(\, ^i\!\theta^b_{11} + 2 \sum_k \big((\overline{d}_{ki}-c_{ik})\psi_2^k - (d_{ki}-\overline{c}_{ik})\overline{\psi}_2^k \big) \Big) \delta_{ij},\]
and thus the Bismut curvature is given by
\[ \Theta^b = d\theta^b - \theta^b \theta^b = \!\begin{bmatrix} dS & 0 \\ 0 & 0 \end{bmatrix}\!.\]
Here we have,
\[dS_{ij} = \big(\kappa_i - 8\lambda_i(\sum_k|d_{ki}|^2 -1)\big) \varphi_1^i \overline{\varphi}_1^i \delta_{ij}.\]
Therefore, it is easy to see $(\varphi_1^i, \psi_2^i) \Theta^b =0$, and thus the pluriclosed twisted product is {\em BKL}.
\end{proof}

Hence we have proved the first part of Theorem \ref{thm4}. As to the second part, let $(M^4,g)$ be a complete, simply-connected full {\em BKL} manifold. It follows from Theorem \ref{thm1} that the $B$-rank $r$ must be either $2$ or $3$. If $r=3$, $M^4$ is Bismut flat by Theorem \ref{thm2}. If $r=2= \frac{n}{2}$, it turns out from the proof above to be the pluriclosed twisted product of two non-K\"ahler {\em BKL} surfaces. The proof of Theorem \ref{thm4} is completed.
\vsv

For the remaining part of this section, we will discuss the {\em pluriclosed twisted product} for factors in higher dimensions. It turns out that it is actually a phenomenon of {\em BKL} surfaces, which does not generalize to higher dimensions, in order to preserve the pluriclosedness. Let $(N_1^n,h_1)$, $(N_2^m,h_2)$ be two {\em BKL} manifolds with $A>0$. Write $M^{n+m}= N_1\times N_2$ for the product complex manifold and $h=h_1\times h_2$ for the product metric. Of course $h$ is {\em BKL}. Denote by $\varphi$, $\varphi'$ the local coframes dual to $\phi$-compatible frames $e,e'$ on $N_1$, $N_2$, respectively. Let
$$ \psi_1 = u_1 \varphi_n + u_2 \varphi_m', \ \ \ \psi_2 = v_1 \varphi_n + v_2 \varphi_m'$$
where $u_1$, $u_2$, $v_1$, $v_2$ are constants, satisfying
\[ \det\! \begin{bmatrix}u_1 & u_2 \\
                         v_1 & v_2 \end{bmatrix}\! \neq 0,\quad  \]
Define a Hermitian metric $g$ on the product complex manifold $M=N_1\times N_2$ by letting
$$ \{ \varphi_1 , \ldots , \varphi_{n-1}, \psi_1, \varphi_1', \ldots , \varphi_{m-1}', \psi_2\} $$
to be its unitary coframe. Clearly, $g$ is well defined. When $\psi_1$ is not perpendicular to $\psi_2$, namely, when $\beta := u_1\overline{u}_2 + v_1\overline{v}_2\neq 0$, $g$ is not a product metric on $M$.

Since $\eta = \lambda \varphi_n$, $\eta' = \lambda ' \varphi_m'$, we have
$$ \psi_1 \overline{\psi}_1 + \psi_2 \overline{\psi}_2 = \frac{\alpha}{\lambda^2} \eta \overline{\eta} + \frac{\alpha '}{\lambda'^2} \eta' \overline{\eta'} + \frac{\beta }{\lambda \lambda '} \eta \overline{\eta'} + \frac{\overline{\beta}}{\lambda \lambda '} \eta' \overline{\eta}, $$
where $\alpha = |u_1|^2+ |v_1|^2$ and $\alpha' = |u_2|^2+ |v_2|^2$. From this, we see that
$$ \omega_g = \omega_h  + \sqrt{\!-\!1} \big( \frac{\alpha -1}{\lambda^2} \eta \overline{\eta} + \frac{\alpha ' -1}{\lambda'^2} \eta' \overline{\eta'} + \frac{\beta }{\lambda \lambda'} \eta \overline{\eta'} + \frac{ \overline{\beta}}{ \lambda \lambda'}  \eta' \overline{\eta} \big) .$$
Taking $\partial \overline{\partial}$, we get that $\partial \overline{\partial}\omega_g =0$ if and only if
\begin{equation}
(\alpha -1) \partial \overline{\partial} (\eta \overline{\eta}) = (\alpha' -1) \partial \overline{\partial} (\eta' \overline{\eta'})=  \partial \overline{\partial} \big( \beta \eta \overline{\eta'} + \overline{\beta} \eta' \overline{\eta} \big) = 0.
\end{equation}
By using \cite[Lemma 15]{ZZ}, this means that
$$ (\alpha -1)\mbox{Re} (a_i \overline{a_k}) = (\alpha' -1)\mbox{Re} (a_j' \overline{a_{\ell}'}) = \mbox{Re} (\beta \overline{a_i} a_j') =0, \ \ \ \forall \ 1\leq i\neq k\leq n, \ \forall \ 1\leq j\neq \ell \leq m $$
where $a_i$ and $a_j'$ are the eigenvalues of $\frac{1}{\lambda}\phi$ and $\frac{1}{\lambda'}\phi'$ of $h_1$ and $h_2$, respectively. Since $\sum_i a_i =\lambda$ and $\sum_j a_j'=\lambda'$, by summing over $i$ and $j$ in the third equation above, we get
$$ \mbox{Re} (\beta ) =0, \ \ \ \mbox{Re} (\beta \overline{a_i})=0, \ \ \ \mbox{Re} (\beta a_j') = 0. $$
Now if $\beta\neq 0$, then $\beta=ic$ for some real number $c \neq 0$, and all $a_i$ and all $a_j'$ are real. Note that
for a {\em BKL} manifold in dimension $3$ with $A > 0$, the two eigenvalues $a_1$ and $a_2$ cannot be all real as shown in \cite{YZZ}. So
the above metric $g$ cannot be pluriclosed if one of factor is of dimension $3$ with $A>0$, unless $\beta = 0$. In particular, there is no twisted pluriclosed product structure for {\em BKL} fivefolds with $A > 0$.

In general, it would be an interesting question to understand all {\em BKL} manifolds $(M^n,g)$ with $A>0$, such that all their $\frac{1}{\lambda} \phi$-eigenvalues $a_i$ are real. Recall that a Hermitian manifold is called {\em locally conformally balanced,} if $d(\eta +\overline{\eta})=0$. For {\em BKL} manifolds, this is equivalent to all $a_i$'s being real. So the above question can be rephrased as understanding all full {\em BKL} manifolds that are locally conformally balanced.

Continue with our discussion on twisted pluriclosed product structure, note that when $\beta =0$, $g$ is the product metric $h_1'\times h_2'$, where $h_i'$ is a modification of $h_i$ in its $\eta$-direction. In general dimensions, for {\em BKL} manifolds satisfying the condition $\partial \overline{\partial}  (\eta \overline{\eta})=0$, one can always scale the $\eta$ direction to get a one parameter family of {\em BKL} metrics.

By \cite[Lemma 15]{ZZ}, the condition  $\partial \overline{\partial}  (\eta \overline{\eta})=0$ is equivalent to
\begin{equation}
 \mbox{Re} (a_i \overline{a_k}) = 0 \ \ \ \ \forall \ i\neq k .
 \end{equation}
Note that this means that $\frac{a_i}{|a_i|} = \pm \sqrt{-1} \frac{a_k}{|a_k|}$ when $i\neq k$ and $a_ia_k\neq 0$. So there can be at most two non-zero $a_i$'s, which will force the dimension to be at most $4$ if $A>0$, and the dimension $4$ case is also twisted product of two {\em BKL} surfaces by Theorem \ref{thm4}.

Now let $(M^n,g)$ be a {\em BKL} manifold satisfying $\partial \overline{\partial}  (\eta \overline{\eta})=0$ and $e$ be a $\phi$-compatible frame. One can always scale the $e_n$ term and still get {\em BKL} metrics. In other words, let us define a Hermitian metric $h$ on $M^n$ with local unitary frame $\{ e_1, \ldots , e_{n-1}, \frac{1}{t}e_n\}$, where  $t>0$ is a constant. Alternatively, $h$ can be expressed as
\begin{equation}
 \omega_h = \sqrt{\!-\!1} \{ \varphi_1 \overline{\varphi}_1 + \cdots + \varphi_{n-1} \overline{\varphi}_{n-1} + t^2\varphi_n \overline{\varphi}_n \}, \label{eq:omegah}
\end{equation}
where $\varphi$ is the coframe dual to $e$. We will call this new metric $h$ an {\em $\eta$-scaling metric} of $g$. Since $\eta = \lambda \varphi_n$, we have
$$ \omega_h - \omega_g = \frac {\sqrt{\!-\!1}\,(t^2-1)}{\lambda ^2} \eta \,\overline{\eta}. $$
Therefore we see that $\omega_h$ is pluriclosed. It is actually also {\em BKL}.

\begin{proposition}
Let $(M^n,g)$ be a BKL manifold satisfying $\partial \overline{\partial}  (\eta \overline{\eta})=0$. Then for any positive constant $t$, the $\eta$-scaling metric $h$ defined by (\ref{eq:omegah}) is BKL.
\end{proposition}

\begin{proof}
Let $\tilde{\varphi}_i=\varphi_i$ for $1\leq i \leq n-1$, and $\tilde{\varphi}_n= t\varphi_n$. Then $\tilde{\varphi}$ becomes a local unitary coframe for $h$. Denote by
$$ P =\! \begin{bmatrix} I_{n-1} & 0 \\ 0 & t \end{bmatrix}\!, \quad v =\! \begin{bmatrix} a_1\varphi_1 \\ \vdots \\ a_{n-1}\varphi_{n-1} \end{bmatrix} \! . $$
Then the Chern connection matrix $\theta^c$ and Chern torsion vector $\tau$ under $\varphi$ are
$$ \theta^c = \theta^b -2 \gamma =\!\begin{bmatrix} \theta^b_1 & 0 \\ 0 & 0 \end{bmatrix}\!  - 2 \!\begin{bmatrix} \gamma_1 & \overline{v}  \\ -\,^t\!v & 0 \end{bmatrix}\! , \quad \tau = \!\begin{bmatrix} \tau_1 \\ 0  \end{bmatrix}\!  ,$$
where we used the subscript $1$ to denote the $(n-1)\times (n-1)$ block or the first $(n-1)$ rows. Let
$$ \Gamma =\! \begin{bmatrix} \gamma_1 & t\overline{v} \\ -t \,^t\!v & 0 \end{bmatrix}\! , $$
then we have
\begin{eqnarray*}
 d\tilde{\varphi} & = &  P d\varphi \ = \ P ( - \,^t\!\theta \varphi + \tau ) \ = \ - \,^t\!(P^{-1} \theta P) \tilde{\varphi} + P\tau \\
 & = & - \,^t\!(P^{-1} \theta^b P - 2 P^{-1} \gamma P ) \tilde{\varphi} + P\tau \\
 & = & - \,^t\!( \theta^b  - 2 P^{-1} \gamma P ) \tilde{\varphi} + \tau \\
 & = & - \,^t\!( \theta^b  - 2 \Gamma ) \tilde{\varphi} + 2 \,^t\!( P^{-1} \gamma P - \Gamma ) \tilde{\varphi} +  \tau \\
 & = & - \,^t\!( \theta^b  - 2 \Gamma ) \tilde{\varphi} + 2(t-\frac{1}{t})v\tilde{\varphi}_n + \tau.
 \end{eqnarray*}
From this, we know that the Chern connecton matrix $\tilde{\theta}$ and torsion vector $\tilde{\tau}$ of $h$ under $\tilde{\varphi}$ are
$$ \tilde{\theta}^c = \theta^b - 2\Gamma , \ \ \ \ \tilde{\tau} = 2(t-\frac{1}{t})v\tilde{\varphi}_n + \tau. $$
Hence the torsion components of $h$ are
$ \tilde{T}^j_{ik} = T^j_{ik}$ for $i,k<n$ and $\tilde{T}^j_{in} = \tilde{a}_i \delta_{ij}$ where $\tilde{a}_i = ta_i$. Thus, the $\gamma$ matrix of $h$ is $\tilde{\gamma }=\Gamma$, and the Bismut connection matrix of $h$ is
$$ \tilde{\theta}^b = \tilde{\theta} + 2 \tilde{\gamma} = \theta^b - 2\Gamma + 2\Gamma = \theta^b. $$
So $\tilde{\Theta}^b = \Theta^b$, which is block diagonal with lower right corner being zero. Therefore $\,^t\!\tilde{\varphi} \, \tilde{\Theta}^b =0$, and $h$ is {\em BKL}.
\end{proof}

Note that we have proved here that $\tilde{\lambda} = t \lambda $, so $\tilde{\eta} = t^2\eta$. Also, $\tilde{\Theta}^b = \Theta^b$, under $\tilde{e}$ and $e$, respectively. In particular, $h$ will be Bismut flat if and only if $g$ is so.

\vs

\section{{\em BKL} manifolds of dimension $5$.}\label{BKL5D}
Let us discuss the $5$-dimensional case. Let $(M^5,g)$ be a complete, simply-connected {\em BKL} manifold with $A>0$. The $B$-rank $r=r_{B}$ can only be $3$ or $4$ by Theorem \ref{thm1}. The $r=4$ case implies that $g$ is Bismut flat by Theorem \ref{thm2}, so we will assume that $r=3$ from now on. Let $e$ be a strictly $\phi$-compatible frame, with $a_1a_2a_3\neq 0$ and $a_4=0$. For convenience, we may set $b_{5i} = a_i$ so that the entries of $\hat{b}$ can be unified as $(b_{\alpha i})$, where $4 \leq \alpha \leq 5$, $1 \leq i \leq 3$, which are global constants on $M^5$ by Remark \ref{cst}. It follows from Lemma \ref{lemma2} that $\mbox{rank}\,(\hat{b})=2$. We also know from Remark \ref{cst} that $e_4$ and $e_5$ are globally defined, $\nabla^b$-parallel, holomorphic vector fields on $M^5$. It holds that $\theta^b$ is diagonal, and thus $\Theta^b$ is also diagonal, with $\Theta^b_{44}=\Theta^b_{55}=0$.

If $T^j_{ik}\neq 0$ for some $i,j,k\leq 3$, then Lemma \ref{lemma3} implies that, for $4 \leq \alpha \leq 5$,
\[b_{\alpha j} = b_{\alpha i} + b_{\alpha k},\]
and thus $i$, $j$ and $k$ are necessarily distinct. Without loss of generality, we may assume that $T^3_{12} \neq0$. Then it holds that $T^1_{23} = 0$ and $T^2_{13} =0$. If not, say $T^{1}_{23} \neq 0$, Lemma \ref{lemma3} indicates, for $4 \leq \alpha \leq 5$,
\[b_{\alpha 1} = b_{\alpha 2} + b_{\alpha 3},\]
which implies $b_{\alpha 2}=0$, and it contradicts with $a_2 =b_{52} \neq 0$. Then
let $i=1,k=2$ in the equality \eqref{eq:Psum=0}, it yields
\[|T^3_{12}|^2 = - \sum_{\alpha=4}^5 2 \mbox{Re} (b_{\alpha 1} \overline{b}_{\alpha 2}).\]
Similarly, let $i=1,k=3$ and $i=2,k=3$ in the equality \eqref{eq:Psum=0} respectively and we get
\begin{gather*}
|T_{12}^3|^2 = \sum_{\alpha =4}^5 2 \mbox{Re} (b_{\alpha 1} \overline{b}_{\alpha 3}),\\
|T_{12}^3|^2 = \sum_{\alpha =4}^5 2 \mbox{Re} (b_{\alpha 2} \overline{b}_{\alpha 3}).
\end{gather*}
These three equalities implies, together with $b_{\alpha 3} = b_{\alpha 1} + b_{\alpha 2}$,
\[|T_{12}^3|^2 = \sum_{\alpha} |b_{\alpha 1}|^2 = \sum_{\alpha} |b_{\alpha 2}|^2 = \sum_{\alpha} |b_{\alpha 3}|^2.\]
Hence, $|T_{12}^3|$ is a global constant. Then we may rotate $e_3$ while fixing others so that $T^3_{12} = |T^3_{12}|$.
Then for any tangent vector $X$, we have
$$\begin{aligned}
0 = T^3_{12,X} &= X(T^3_{12}) - T^3_{q2}\theta^b_{1q}(X) - T^3_{1q} \theta^b_{2q}(X) + T^q_{12}\theta^b_{q3}(X) \\
               &= -T^3_{12} \{ \theta^b_{11}(X) + \theta^b_{22}(X) - \theta^b_{33}(X) \},
\end{aligned}$$
which implies
\[ \theta^b_{11} + \theta^b_{22} = \theta^b_{33},\]
and thus
$$ \Theta^b_{11}=\Theta^b_{22}=\Theta^b_{33}=0.$$
So $\Theta^b_{qq}=0$ for all $1\leq q\leq 5$ and the manifold is Bismut flat.

After this, we will assume that
\begin{equation}
T^j_{ik}=0 \ \ \ \ \forall \ i,j,k \leq 3.\label{eq:allisolated}
\end{equation}
This means that the only possibly non-zero torsion components are $b_{\alpha i}$, where $1 \leq i \leq 3$ and $4 \leq \alpha \leq 5$. The the equality \eqref{eq:Psum=0} implies
\begin{equation}
2\mbox{Re} \sum_{\alpha=4}^5 ( b_{\alpha i} \overline{b}_{\alpha k}) =0, \ \ \ \ 1\leq i <k\leq 3 .\label{eq:abik}
\end{equation}
Also, by the property $B=\phi + \phi^{\ast}$ for {\em BKL} manifolds, we get
\begin{equation}
\lambda ( b_{5i} + \overline{b}_{5i}) = 2 \sum_{\alpha=4}^5|b_{\alpha i}|^2, \ \ \ \ \  1\leq i\leq 3. \label{eq:abi}
\end{equation}
Note that $\theta^b$ is diagonal, with $\theta^b_{44}=\theta^b_{55}=0$. It follows that the Levi-Civita covariant derivatives
$$ \nabla e_i = \nabla^b e_i - \gamma (e_i) + \overline{\theta_2}(e_i) = \theta^b_{ii} e_i +\sum_{q,k} \{\overline{T^i_{qk}} \overline{\varphi}_k e_q -T^q_{ik}\varphi_k e_q + T^k_{iq} \overline{\varphi}_k \overline{e}_q \} ,$$
which yields, for $1 \leq i \leq 3$ and $4 \leq \alpha \leq 5$,
\begin{eqnarray*}
\nabla e_i & = & \xi_i e_i - \overline{\varphi}_i (X_i-\overline{X}_i ) ,  \\
\nabla e_\alpha &= &\sum_{i=1}^3 b_{\alpha i} (\varphi_i e_i - \overline{\varphi}_i \overline{e}_i),  \quad \mbox{where}\\
X_i & = & \sum_{\alpha=4}^5 \overline{b}_{\alpha i} e_\alpha, \\
\xi_i & = & \theta^b_{ii} - \sum_{\alpha=4}^5 b_{\alpha i} \varphi_{\alpha} + \sum_{\alpha=4}^5 \overline{b}_{\alpha i} \overline{\varphi}_{\alpha}.
\end{eqnarray*}
It is clear from $\mbox{rank}\,(\hat{b})=2$ that $\mbox{span}\{X_1,X_2,X_3\}=\mbox{span}\{e_4,e_5\}=N$.
Decompose $e_i$ into the real and imaginary parts, for $1 \leq i \leq 5$, as
\[e_i = \frac{1}{\sqrt{2}}(\eps_i - \sqrt{\!-\!1} \eps_{i^*}),\]
where $\eps_{i^*}=\eps_{i+5}$. Note that, for each $1 \leq i \leq 3$,
\[ -\sqrt{-1}(X_i - \overline{X}_i) = -\sqrt{-1} \sum_{\alpha=4}^5 (\overline{b}_{\alpha i} e_\alpha - b_{\alpha i} \overline{e}_{\alpha}), \]
and its length square is $2\sum\limits_{\alpha=4}^5|b_{\alpha i}|^2 = B_{i\bar{i}}:= B_i$. Denote by $Y_i$ the unit vector in the direction of $-\sqrt{-1}(X_i - \overline{X}_i)$, that is, $-\sqrt{\!-\!1}(X_i-\overline{X}_i ) = \sqrt{B_{i}} \,Y_i.$
Then the equality (\ref{eq:abik}) indicates that $Y_1$, $Y_2$, $Y_3$ are mutually perpendicular in the real underlying space of $N$
$$ N_{\mathbb R} = \mbox{span}\{ \eps_4, \eps_{4^*}, \eps_5, \eps_{5^*} \} .$$
Let $Y_4$ be the unique unit vector in $N_{\mathbb R}$ such that $\{ Y_1, Y_2, Y_3, Y_4\}$ forms a positive orthonormal basis of $N_{\mathbb R}$. It is easy to see that $Y_4$ can be expressed as
$$ Y_4 = \frac{-\sqrt{\!-\!1}}{\sqrt{2}} \sum_{\alpha=4}^5 (\overline{b}_{\alpha 4} e_{\alpha} - b_{\alpha 4} \overline{e}_\alpha), $$
where $b_{\alpha 4}$ are complex numbers such that $ \sum\limits_{\alpha=4}^5 \mbox{Re} (b_{\alpha 4} \overline{b}_{\alpha i})  =0$ for $1\leq i \leq 3$, $4\leq \alpha \leq 5$ and $\sum\limits_{\alpha=4}^5|b_{\alpha 4}|^2=1$. It follows, for each $1\leq i\leq 3$, that
\begin{eqnarray*}
\nabla Y_4 & = &- \sqrt{\!-\!2} \sum_{\begin{subarray}{c} 1 \leq k \leq 3 \\ 4 \leq \alpha \leq 5\end{subarray}}
\big\{ \mbox{Re} (b_{\alpha 4}\overline{b}_{\alpha k} ) \cdot (\varphi_k e_k - \overline{\varphi}_k \overline{e}_k) \big\} \ = \ 0, \\
\nabla Y_i & = & \frac{-2\sqrt{\!-\!1}}{\sqrt{B_i}} \sum_{\begin{subarray}{c} 1 \leq k \leq 3 \\ 4 \leq \alpha \leq 5\end{subarray}}
\big\{ \mbox{Re} (b_{\alpha i}\overline{b}_{\alpha k}) \cdot (\varphi_k e_k - \overline{\varphi}_k \overline{e}_k ) \big\} \\
& = &  -\sqrt{\!-\!B_i} (\varphi_i e_i - \overline{\varphi}_i \overline{e}_i),  \\
\nabla e_i & = & \xi_i e_i -  \sqrt{\!-\! B_{i}} \, \overline{\varphi}_i\, Y_i.
\end{eqnarray*}
It is clear that the distribution $L_i = \mbox{span} \{ \eps_i, \eps_{i^*}, Y_i\}$ is parallel with respect to $\nabla$, hence on $M^5$ it gives a de Rham decomposition factor, and by
$$ \nabla_{\eps_i} Y_i = \sqrt{B_{i}} \,\eps_{i*}, \ \ \ \nabla_{\eps_{i^*}} Y_i = - \sqrt{B_{i}} \,\eps_{i}, $$
we know that $\frac{1}{\sqrt{B_{i}}} \nabla Y_i$ is equal to the complex structure $J$ on $H_i:= \mbox{span}\{ \eps_i , \eps_{i^*}\}$, the orthogonal complement of $Y_i$ in $L_i$.  Also, by the fact $\nabla_{Y_i}Y_i=0$, $Y_i$ is a Killing vector field on $L_i$. So, as in \cite{YZZ}, $L_i$ is a Sasakian $3$-manifold with Reed vector field $Y_i$. Note that the complex structure $J\big|_{N_{\mathbb{R}}}$ forms a constant skew-symmetric and orthogonal matrix with respect to the Reed vector fields $\{Y_1,Y_2,Y_3,Y_4\}$, whose determinant is necessarily $1$. Therefore, $(M^5,g)$ is isometric to the metric product $L_1\times L_2 \times L_3 \times {\mathbb R}$, where $L_i$ are Sasakian $3$-manifolds and the last factor is generated by the $\nabla$-parallel vector field $Y_4$. The compatible complex structure $J$ on the distribution generated by the four Reeb vector fields is determined by a $4\times4$ skew symmetric and orthogonal matrix, while on $H_i$ it amounts to $\frac{1}{\sqrt{B_{i}}} \nabla Y_i$, for $1 \leq i \leq 3$. This makes $L_1\times L_2 \times L_3 \times {\mathbb R}$ into a Hermitian manifold and motivates the following definition

\begin{definition}\label{mprod}
Given $r$ Sasakian $3$-manifolds $(L_1,h_1),\cdots,(L_r,h_r)$ and a skew symmetric and orthogonal matrix $D=(d_{ij})$ of the size $r+s$, where  $r$, $s$ are nonnegative integers with $r+s$ being even, we can define the {\em multiple product of  Sasakian $3$-manifolds}, which is the Riemannian product manifold $\Pi_{i=1}^r L_i \times \mathbb{R}^s$, endowed with the product metric and a compatible almost complex structure $J$ described below:

For each $1\leq i\leq r$, denote by $Y_i$ the Reeb vector field on $L_i$ and write $Y_{r+j}=\frac{\partial}{\partial x_{j}}$, where $(x_1, \ldots , x_s)$ is the standard coordinate of on $\mathbb{R}^{s}$, for $1 \leq j \leq s$, with $Y_{r+j}$ being null in case $s=0$. When restricted on the distribution generated by $\{ Y_{1},\ldots,Y_{r+s}\} $, $J$ is determined by $D$ in the manner $JY_{k}=\sum_{\ell =1}^{r+s} d_{k\ell }Y_{\ell}$, while on each $H_i$ it is equal to $\frac{1}{c_i} \nabla Y_i$. Here $H_i$ is the orthogonal complement of $Y_i$ in $L_i$, $\nabla$ is the Levi-Civita connection of $(L_i,h_i)$, and $c_i$ is some positive constant. This $J$ is actually integrable as we shall see below, so it makes $\Pi_{i=1}^r L_i \times \mathbb{R}^s$ into a Hermitian manifold.
\end{definition}

It is easy to see that multiple product of Sasakian $3$-manifolds with $r=2$ and $s=0$ are just the standard Hermitian structure on the product of two Sasakian $3$-manifolds as in \cite{Belgun12} and \cite{YZZ}.

\begin{proposition}
The multiple product of Sasakian $3$-manifolds $\Pi_{i=1}^r L_i \times \mathbb{R}^s$ is BKL.
\end{proposition}

\begin{proof}
Let $(\eps_i,\eps_{i^*},Y_i)$ be an orthonormal frame on the Sasakian $3$-manifold $(L_i,h_i)$ for $1 \leq i \leq r$, where $Y_i$ is the Reeb vector field on $L_i$ and $\frac{1}{c_i} \nabla_{\eps_i} Y_i=\eps_{i^*}$. By the definition of Sasakian $3$-manifold, $\frac{1}{c_i} \nabla Y_i$ determines an orthogonal integrable complex structure on  $H_i= \mbox{span}\{\eps_i,\eps_{i^*}\}$. Let
\[e_i = \frac{1}{\sqrt{2}}(\eps_i - \sqrt{-1} \eps_{i^*})\]
and we get the following structure equation of $(L_i,h_i)$ under the frame $(Y_i,e_i,\overline{e}_i)$
\[ \nabla \! \begin{bmatrix} Y_i \\ e_i \\ \overline{e}_i \end{bmatrix}=
\begin{bmatrix} 0 & c_i \sqrt{-1}\varphi_i & - c_i \sqrt{-1} \overline{\varphi}_i \\
c_i \sqrt{-1} \overline{\varphi}_i & \overline{\alpha}_i & 0 \\
-c_i \sqrt{-1} \varphi_i & 0 & \alpha_i \end{bmatrix}\!\!
\begin{bmatrix} Y_i \\ e_i \\ \overline{e}_i \\ \end{bmatrix}\!\!,\]
where $\alpha_i$ is a $1$-form satisfying $\alpha_i + \overline{\alpha}_i=0$. Denote the dual coframe of $(Y_i,e_i,\overline{e}_i)$  by $(\phi_i,\varphi_i,\overline{\varphi}_i)$, then it follows that
\begin{equation}\label{dssk}
d \! \begin{bmatrix} \phi_i \\ \varphi_i \\ \overline{\varphi}_i \end{bmatrix}\!=
\nabla \! \begin{bmatrix} \phi_i \\ \varphi_i \\ \overline{\varphi}_i \end{bmatrix} \!=\!
\begin{bmatrix} 0 & -c_i \sqrt{-1} \overline{\varphi}_i & c_i \sqrt{-1} \varphi_i \\
-c_i \sqrt{-1} \varphi_i & \alpha_i & 0 \\
c_i \sqrt{-1} \overline{\varphi}_i & 0 & \overline{\alpha}_i \end{bmatrix}\!\!
\begin{bmatrix} \phi_i \\ \varphi_i \\ \overline{\varphi}_i \\ \end{bmatrix}\!\!.
\end{equation}
The second equality of \eqref{dssk} above implies that $d \alpha_i \wedge \varphi_i =0$ and thus $d \alpha_i \wedge \overline{\varphi}_i=0$. It indicates that
\begin{equation}\label{dalpha} d \alpha_i = \kappa_i \varphi_i \overline{\varphi}_i,\end{equation}
for some local real-valued function $\kappa_i$.

On the distribution generated by $\{ Y_{1},\ldots,Y_{r+s}\} $, $J$ is determined by the equation
\[ JY_{k}=\sum_{\ell =1}^{r+s} d_{k\ell}Y_{\ell },\]
where $D=(d_{ij})$ is a skew-symmetric and orthogonal matrix of the size $r+s$. Write $r+s=2m$. It follows from linear algebra that there exists a matrix $P=(p_{k \ell})\in SO(2m)$ such that
\[ P^{-1}DP=\!\begin{bmatrix} 0 & I_{m} \\ -I_{m} & 0  \end{bmatrix}\!,\]
where $I_{m}$ is the identity matrix of the size $m$. Denote by $Y$ and $Z$ the column vector $^t\!(Y_1,\cdots,Y_{r+s})$ and $^t\!(Z_1,\cdots,Z_{r+s})$, respectively, where $Z$ is set to be $PY$. Extend $\{\phi_i\}_{i=1}^r$ above to $\{\phi_k\}_{k=1}^{r+s}$, abbreviated as $\phi$, such that it is dual to $Y$, and write the dual basis of $Z$ as $\psi$. It is clear that
\[\quad \psi = (^t\!P)^{-1}\phi=P\phi.\]
For $1 \leq i,j \leq r$, $1 \leq \alpha, \beta \leq \frac{r+s}{2}$, we define
\[ \tilde{e}_{\alpha} = \frac{1}{\sqrt{2}}(Z_{\alpha}- \sqrt{-1} Z_{\alpha^*}),\quad
\tilde{\varphi}_{\alpha} = \frac{1}{\sqrt{2}}(\psi_{\alpha} + \sqrt{-1}\psi_{\alpha^*}),\]
where $\alpha^* = \frac{r+s}{2} + \alpha$, and thus $(e_i,\tilde{e}_{\alpha})$ forms a natural unitary frame on the multiple product of $3$-Sasakian manifolds $\Pi_{i=1}^r L_i \times \mathbb{R}^s$, with $(\varphi_i,\tilde{\varphi}_{\alpha})$ being its dual coframe. Then, for $1 \leq \alpha \leq \frac{r+s}{2}$, it yields from \eqref{dssk} that
\begin{align*}
d \tilde{\varphi}_{\alpha} &=  \frac{1}{\sqrt{2}}(d\psi_{\alpha} + \sqrt{-1}d\psi_{\alpha^*}) \\
 &= \frac{1}{\sqrt{2}} \sum_{k=1}^{r+s} (p_{\alpha k} + \sqrt{-1}p_{\alpha^* k}) d\phi_k \\
 &= \sqrt{-2} \sum_{i=1}^{r} c_i (p_{\alpha i} + \sqrt{-1}p_{\alpha^* i}) \varphi_i \overline{\varphi}_i,
\end{align*}
where we use the fact that $d \phi_k=0$ for $k > r$. It also follows from \eqref{dssk} that
\begin{align*}
d \varphi_i &= \Big\{\alpha_i + c_i \frac{\sqrt{-1}}{\sqrt{2}} \sum_{\alpha=1}^{\frac{r+s}{2}}
\big((p_{\alpha i} - \sqrt{-1}p_{\alpha^* i}) \tilde{\varphi}_\alpha + (p_{\alpha i} + \sqrt{-1}p_{\alpha^* i}) \overline{\tilde{\varphi}}_\alpha \big)  \Big\}\varphi_i,
\end{align*}
Hence we get, under the unitary frame $(\varphi_i, \tilde{\varphi}_{\alpha})$
\[d\!\begin{bmatrix} \varphi_i \\ \tilde{\varphi}_\alpha \end{bmatrix}\!=\!\begin{bmatrix} R & Q \\ -^t\!\overline{Q} & 0 \end{bmatrix}\!
\begin{bmatrix} \varphi_i \\ \tilde{\varphi}_\alpha \end{bmatrix}\! + \!\begin{bmatrix} \beta_i \\ 0 \end{bmatrix}\!,\]
where
\begin{gather*}
R_{ij}= \Big(\alpha_j + \frac{c_j \sqrt{-1}}{\sqrt{2}}
\sum_\alpha \big( (p_{\alpha j} - \sqrt{-1} p_{\alpha^* j}) \tilde{\varphi}_{\alpha}
+(p_{\alpha j} + \sqrt{-1} p_{\alpha^* j}) \overline{\tilde{\varphi}}_{\alpha} \big) \Big) \delta_{ij}, \\
Q_{i \alpha} = -\sqrt{-2}c_i (p_{\alpha i} -\sqrt{-1} p_{\alpha^* i}) \varphi_i,\quad
\beta_i = \sqrt{-2}c_i \sum_\alpha (p_{\alpha i} -\sqrt{-1} p_{\alpha^* i}) \varphi_i \tilde{\varphi}_{\alpha},
\end{gather*}
and we know that $J$ is integrable as no $(0,2)$-form appears in the right hand side of the equations above. Hence the Chern connection matrix $\theta^c$ and the components of Chern torsion vector $\tau$ under the frame $(\varphi_i, \tilde{\varphi}_{\alpha})$ are
\[ \theta^c =\! \begin{bmatrix} \overline{R} & \overline{Q} \\ -^t\!Q & 0 \end{bmatrix}\!, \quad \tau^i = \beta_i, \quad \tau^{\alpha}=0. \]
It indicates that the only possibly non-zero torsion components are
\[T^i_{j \alpha} = \frac{c_j \sqrt{-1}}{\sqrt{2}}(p_{\alpha j} -\sqrt{-1} p_{\alpha^* j}) \delta_{ij},\]
which implies that the tensor $\gamma$ is given by
\begin{gather*}
\gamma_{ij} = c_j \frac{\sqrt{-1}}{\sqrt{2}} \delta_{ij}\sum_{\alpha}
\big((p_{\alpha j} - \sqrt{-1}p_{\alpha^* j}) \tilde{\varphi}_\alpha + (p_{\alpha j} + \sqrt{-1}p_{\alpha^* j}) \overline{\tilde{\varphi}}_\alpha \big),\\
\gamma_{\alpha i} = - \frac{c_i \sqrt{-1}}{\sqrt{2}} (p_{\alpha i} -\sqrt{-1} p_{\alpha^* i}) \varphi_i,\quad
\gamma_{i \alpha} = -\overline{\gamma_{\alpha i}},\quad \gamma_{\alpha \beta}=0.
\end{gather*}
Hence the Bismut connection matrix $\theta^b$ is given by
\[ \theta^b = \theta^c + 2\gamma = \!\begin{bmatrix} S & 0 \\ 0 & 0 \end{bmatrix}\!,\]
where
\[S_{ij} = \Big( - \alpha_j + \frac{c_j \sqrt{-1}}{\sqrt{2}}
\sum_\alpha \big( (p_{\alpha j} - \sqrt{-1} p_{\alpha^* j}) \tilde{\varphi}_{\alpha}
+(p_{\alpha j} + \sqrt{-1} p_{\alpha^* j}) \overline{\tilde{\varphi}}_{\alpha} \big) \Big) \delta_{ij},\]
and thus the Bismut curvature is given by
\[ \Theta^b = d\theta^b - \theta^b \theta^b = \!\begin{bmatrix} dS & 0 \\ 0 & 0 \end{bmatrix}\!.\]
Here we have, from \eqref{dalpha} and $P=(p_{k \ell}) \in SO(2m)$,
\[dS_{ij} = (2c_j^2 - \kappa_j)\varphi_j \overline{\varphi}_j \delta_{ij}.\]
Therefore, it is easy to see that $(\varphi_i, \tilde{\varphi}_{\alpha}) \Theta^b =0$, and thus the multiple product of Sasakian $3$-manifolds is {\em BKL}.
\end{proof}

It follows easily from the above proof that, after an appropriate constant unitary transformation of $\tilde{\varphi}_{\alpha}$, the unitary frame used above can be switched into a $\phi$-compatible one, and by Remark \ref{Abhat} that, when $r \geq \frac{r+s}{2}$ or equivalently when $r \geq s$, the multiple product of Sasakian $3$-manifolds $\Pi_{i=1}^r L_i \times \mathbb{R}^s$ becomes {\em BKL} with $A>0$. The proof of Theorem \ref{thm5} is completed.

\vspace{0.1cm}

We remark that, since any non-K\"ahler {\em BKL} surface is always the Riemannian product $L\times {\mathbb R}$ of a Sasakian $3$-manifold with the line, it can be directly verified that the case $r=s$ of the multiple product of Sasakian $3$-manifolds actually reduces to the pluriclosed twisted product of $r$ non-K\"ahler {\em BKL} surfaces introduced in Definition \ref{p-twisted}. Note that in this case the manifold as a complex manifold is the product of complex surfaces, although this product structure might be different from the metric product structure. When $r>s$, the manifold is in general not a complex product manifold.

\vspace{0.1cm}

K\"ahler manifolds or Bismut flat manifolds are automatically {\em BKL}. Other than those, our discussion above indicates that one can form {\em BKL} manifolds by taking multiple products of Sasakian $3$-manifolds, and up to dimension $5$  this is the only possibility. At present, we do not know what happens when $n\geq 6$. In particular, we do not know whether there could be a {\em BKL} manifold $(M^6,g)$ with $A>0$ and $r=4$, which is neither Bismut flat nor a multiple product of Sasakian $3$-manifolds.

\vspace{0.4cm}

\noindent\textbf{Acknowledgments.} The authors would like to thank Bo Yang for his interest and discussion, which laid down
the foundation of the computation carried out in this paper.

\vs

\end{document}